\documentclass[1 [leqno,11pt]{amsart}
\usepackage{amssymb, amsmath}
\def\ot{\omega^\theta}

\def\ur{u^r}
\def\ut{u^\theta}
\def\uz{u^z}
\let\ve=\varepsilon

\let\pa=\partial
\let\wt=\widetilde

\let\f=\frac
\let\D=\Delta

\def\dive{\mathop{\rm div}\nolimits}
\def\curl{\mathop{\rm curl}\nolimits}

\newcommand{\andf}{\quad\text{and}\quad}
\newcommand{\with}{\quad\text{with}\quad}

\def\cA{{\mathcal A}}
\def\cB{{\mathcal B}}

\def\cD{{\mathcal D}}

\def\cF{{\mathcal F}}
\def\cG{{\mathcal G}}

\def\N{\mathop{\mathbb N\kern 0pt}\nolimits}
\def\R{\mathop{\mathbb R\kern 0pt}\nolimits}

\newcommand{\Rmnum}[1]{\uppercase\expandafter{\romannumeral #1} }

\def\eqdef{\buildrel\hbox{\footnotesize def}\over =}

\newtheorem{thm}{Theorem}[section]
\newtheorem{lem}{Lemma}[section]
\newtheorem{rmk}{Remark}[section]
\newtheorem{cor}{Corollary}[section]
\newtheorem{prop}{Proposition}[section]

\numberwithin{equation}{section}
\begin{document}
\title[Long-time asymptotics of axisymmetric NSE]
{Long-time asymptotics of axisymmetric Navier-Stokes equations
in critical spaces}

\author[Y. Liu]{Yanlin Liu}
\address[Y. Liu]{School of Mathematical Sciences,
Laboratory of Mathematics and Complex Systems,
MOE, Beijing Normal University, 100875 Beijing, China.}
\email{liuyanlin@bnu.edu.cn}

\date{\today}

\maketitle

\begin{abstract} We prove that any global strong solution
to axisymmetric Navier-Stokes equations must eventually become small.
In particular, the limits of $\|\omega^\theta(t)/r\|_{L^1}$ and $\|u^\theta(t)/\sqrt r\|_{L^2}$ are all $0$ as $t$ tends to infinity.
Then by using these, we can refine a series of decay estimates.
In particular, for the global axisymmetric solutions we know till now,
namely the axisymmetric without swirl or with small swirl ones,
these decay estimates hold.
But our result here do not require any smallness conditions beforehand,
thus more general.
\end{abstract}



\section{Introduction}\label{Secintro}
In this paper, we investigate the long-time  behavior
of global solutions to axisymmetric Navier-Stokes equations.
In general, Navier-Stokes equations (N-S for short) in $\R^3$ reads
\begin{equation}\label{1.1}
\left\{
\begin{array}{l}
\displaystyle \pa_t u+u\cdot\nabla u-\Delta u+\nabla P=0,\qquad(t,x)\in\R^+\times\R^3\\
\displaystyle \dive u=0,\\
\displaystyle u|_{t=0} =u_0,
\end{array}
\right.
\end{equation}
where $u(t,x)=(u^1,u^2,u^3)$ stands for the velocity field and $P$
the scalar pressure function of the fluid, which guarantees the
divergence-free condition of the velocity field.

We recall that except the initial data with some special structure,
it is still a big open problem whether or not the system \eqref{1.1} has a
unique global solution with large initial data.
In this paper, we restrict ourselves to the axisymmetric solutions:
$$u(t,x)=u^r(t,r,z)e_r+u^\theta(t,r,z)e_\theta+u^z(t,r,z)e_z,$$
where $(r,\theta,z)$ denotes the cylindrical coordinates in $\R^3$
so that $x=(r\cos\theta,r\sin\theta,z)$,  and
$$e_r=(\cos\theta,\sin\theta,0),\ e_\theta=(-\sin\theta,\cos\theta,0),\ e_z=(0,0,1),\ r=\sqrt{x_1^2+x_2^2}.$$
Then in this case, we can reformulate \eqref{1.1} in cylindrical coordinates as:
\begin{equation}\label{1.2}
\left\{\begin{split}
& \pa_t u^r+(u^r\pa_r+u^z\pa_z) u^r-(\pa_r^2+\pa_z^2+\frac 1r\pa_r-\frac{1}{r^2})u^r-\frac{(u^\theta)^2}{r}+\pa_r P=0,\\
& \pa_t \ut+(u^r\pa_r+u^z\pa_z) \ut-(\pa_r^2+\pa_z^2+\frac 1r\pa_r-\frac{1}{r^2})u^\theta+\frac{u^r u^\theta}{r}=0,\\
& \pa_t u^z+(u^r\pa_r+u^z\pa_z) u^z-(\pa_r^2+\pa_z^2+\frac 1r\pa_r)u^z+\pa_z P=0,\\
& \pa_r u^r+\frac 1r u^r+\pa_z u^z=0,\\
& u|_{t=0} =u_0.
\end{split}\right.
\end{equation}

It is a celebrated result that for the
axisymmetric without swirl case (which means $\ut=0$),
Ladyzhenskaya \cite{La} and independently Ukhovskii
and Yudovich \cite{UY} proved the global well-posedness of \eqref{1.2}.
And later Ne\v{c}as et al. \cite{LMNP} gave a simpler proof.

However, the global well-posedness of \eqref{1.2} with arbitrary
axisymmetric initial data whose swirl part is non-trivial is still open.
To the best of our knowledge, so far we can only establish global well-posedness
when $\ut_0$ is sufficiently small in some sense.
There are numerous works concerning this,
here we only list \cite{CL02, Yau1, Yau2, Chen, Lei, Liu, LZ, ZZT2} for example.

We mention that by denoting $\widetilde{u}\eqdef u^r e_r+u^z e_z$,
it is easy to check that
$$\dive \wt{u}=0 \andf \curl \wt{u} =\omega^\theta e_\theta.$$
Solving this elliptic system shows that $u^r$ and $u^z$ can be uniquely
determined by $\ot$. Hence sometimes it is convenient to rewrite \eqref{1.2} by using
$\ot$ and $\ut$ as:
\begin{equation}\label{otut}
\left\{\begin{split}
&\pa_t \ot+(u^r\pa_r+u^z\pa_z)\ot-(\pa_r^2+\pa_z^2
+\frac 1r\pa_r-\frac{1}{r^2})\omega^\theta
-\frac{u^r \omega^\theta}{r}-\frac{2u^\theta \pa_z u^\theta}{r}=0, \\
&\pa_t \ut+(u^r\pa_r+u^z\pa_z) \ut-(\pa_r^2+\pa_z^2
+\frac 1r\pa_r-\frac{1}{r^2})u^\theta+\frac{u^r u^\theta}{r}=0,\\
&(\ot,\ut)|_{t=0} =(\ot_0, \ut_0).
\end{split}\right.
\end{equation}

In the following, let us focus on the long-time behavior
for the axisymmetric solutions.
First, for the special case without swirl,
Gallay and \v{S}ver\'ak \cite{GS}
proved the global well-posedness of \eqref{otut} with $\ut_0=0$
and $\ot_0\in L^1(\Omega)$.
Moreover, this solution satisfies the long-time asymptotics
\begin{equation}\label{longtimeGS}
\|\ot(t)\|_{L^1(\Omega)}\rightarrow0\quad\text{as}\quad t\rightarrow\infty.
\end{equation}
Here they equip the half-plane $\Omega=\{(r,z)|r>0,z\in\R\}$
with measure $drdz$. Precisely, for any $p\in[1,\infty[$, they denote by $L^p(\Omega)$
the space of measurable function $f$ which verifies
$$
\|f\|_{L^p(\Omega)}\eqdef\bigl(\int_\Omega|f(r,z)|^p drdz\bigr)^{\frac 1p}<\infty.$$
This $L^p(\Omega)$ norm emphasize the similarity between
the axisymmetric case and the 2-D case.
As a comparison, let us recall that the standard 3-D Lebesgue measure is
$rdrdz$.

We remark that \eqref{longtimeGS} can not be derived
from a direct energy estimate or convolution inequalities,
see Remark $\ref{rmk1.2}$ below.
Its proof deeply relies on the fact that
$\omega^\theta/r$ satisfies
\begin{equation}\label{otr1.5}
\pa_t \f{\ot}{r}+(u^r\pa_r+u^z\pa_z) \f{\ot}{r}
-(\Delta+\frac 2r\pa_r)\f{\ot}{r}=0,
\end{equation}
so that we can use the strong maximum principle and Hopf's lemma to prove
that the map $t\mapsto\|\ot(t)\|_{L^1(\Omega)}$ is strictly decreasing,
and hence must have a non-negative limit as $t\rightarrow\infty$.
Then by using a blow-up analysis, they can deduce
that this limit must be $0$.

However, as long as $\ut$ is not identically zero, the elliptic structure
in \eqref{otr1.5} is destroyed. Thus the strategy in \cite{GS} fails
to prove similar asymptotics as \eqref{longtimeGS} for general
axisymmetric solutions. Even for small initial data,
in \cite{LZ} Zhang and I can only prove
\begin{equation}\label{LZdecay}
\sup_{t\in\R^+}\Bigl(t^{1-\frac 1p}\|\ot(t)\|_{L^p(\Omega)}
+t^{1-\frac 3{2\ell}}\|\ot(t)\|_{L^\ell}
+t^{\frac12-\frac 1q}\|\ut(t)\|_{L^q(\Omega)}\Bigr)<\infty,
\end{equation}
for any $p\in[1,\infty],~\ell\in\bigl[\f32,\infty\bigr]$
and $q\in[2,\infty]$. Clearly \eqref{LZdecay} only shows that
$\|\ot(t)\|_{L^1(\Omega)}$ is uniformly bounded in time,
but do not need to decay to $0$ as $t$ tends to infinity.

Our aim in this paper is to establish the same decay estimate
as \eqref{longtimeGS} for general axisymmetric global solutions,
which can be stated as follows:

\begin{thm}\label{thm1.1}
{\sl Consider \eqref{otut} with initial data
$\ot_0\in L^1(\Omega)\cap L^{\frac32},~\ut_0\in L^2(\Omega)$
with $r\ut_0\in L^\infty$. Suppose its associate strong solution
$(\ot,\ut)$ exists globally in time, i.e. there is no finite time singularity
\footnote{For N-S, a point $z_0=(t_0,x_0)$ is called regular
if $u$ is H\"older continuous in a neighborhood of $z_0$.
And the singular point is any point which is not regular.
We have many regularity criteria to rule out the formation of singularities,
such as the famous Ladyzhenskaya-Prodi-Serrin criteria.}.
Then this solution $(\ot,\ut)$ satisfies the following long-time asymptotics
\begin{equation}\label{thm1.1.3}
\lim_{t\rightarrow\infty}\|\ot(t)\|_{L^1(\Omega)\cap L^{\f32}}=0,\andf
\lim_{t\rightarrow\infty}\|\ut(t)\|_{L^2(\Omega)}=0.
\end{equation}
}\end{thm}

\begin{rmk}\label{rmk1.1}

Theorem \ref{thm1.1} can be seen as a weighted space version
of the decay results for general solutions to 3-D N-S
in Besov space proved in \cite{GDP}.
It means that all global
solutions to \eqref{1.2} will eventually become small,
and thus satisfy all the properties linked to small data theory.

In particular, the decay property
\eqref{thm1.1.3} holds for all the global solutions we have mentioned above,
including axisymmetric without swirl case or with small swirl case.
But Theorem \ref{thm1.1} does not require any smallness condition
for the initial data, thus more general.
\end{rmk}

\begin{rmk}\label{rmk1.2}

Let us temporarily focus on 2-D N-S,
which is expected to have similar
behavior as the axisymmetric ones.
The $L^1(\R^2)$ norm of the 2-D vorticity $\omega=\pa_1 u^2-\pa_2 u^1$
can only proved to be non-increasing in time,
but in general does not converge to zero.
Its long-time behavior is described by the self-similar
solutions called Oseen vortices, see \cite{GW1, GW2}.
In particular, if initially $\omega_0$ has a definite sign, then
$\|\omega(t)\|_{L^1(\R^2)}$ is a constant for all time.
Thus the decay for $\|\ot(t)\|_{L^1(\Omega)}$
in the axisymmetric case is unexpected in some sense.

On the other hand, when applying $L^1$ estimate,
the diffusion part only gives
\begin{equation}\label{diffL1}
-\int_{\R^3}\Delta f\cdot{\rm sgn}(f)\,dx\geq0,
\end{equation}
which does not benefit us as much as the $L^p$ estimate with
$1<p<\infty$, namely
$$-\int_{\R^3}\Delta f\cdot|f|^{p-2}\cdot f\,dx
\geq C_p\bigl\|\nabla|f|^{\f p2}\bigr\|_{L^2}^2.$$
One can see this more clearly if the system has elliptic structure to guarantee
the solution to have a definite sign. In this case, \eqref{diffL1} vanishes,
and the $L^1$ norm of the solution does not decay, but remains a constant.
This is the case for 2-D vorticity equation, heat equation, etc.

The decay for $\|\ot(t)\|_{L^1(\Omega)}$ here
deeply relies on the specific structure for axisymmetric N-S,
for instance the appearance of the damping term $r^{-2}\ot$ in the equation,
as well as the boundary conditions on $\{r=0\}$, so that
\eqref{diffL1} would give a boundary term on the artificial boundary
$\{r=0\}$. All these contribute to the decay for $\|\ot(t)\|_{L^1(\Omega)}$.

\end{rmk}

On the other hand, starting from the estimate \eqref{thm1.1.3},
a standard bootstrap argument gives:
\begin{cor}\label{cor1.2}
{\sl For any $p\in[1,\infty],~\ell\in\bigl[\f32,\infty\bigr]$
and $q\in[2,\infty]$,
the unique global solution in Theorem \ref{thm1.1} in addition satisfies
the following refined (comparing to \eqref{LZdecay}) decay estimates:
\begin{equation}
\lim_{t\rightarrow \infty}\Bigl(t^{1-\frac 1p}\|\ot(t)\|_{L^p(\Omega)}
+t^{1-\frac 3{2\ell}}\|\ot(t)\|_{L^\ell}
+t^{\frac12-\frac 1q}\|\ut(t)\|_{L^q(\Omega)}\Bigr)=0.
\end{equation}
}\end{cor}
We refer the readers to the proof of Theorem $1$ in \cite{LZ}
for the details of this procedure. One can also find this argument
in \cite{FS15, GS} for the easier axisymmetric without swirl case.
Clearly Corollary \ref{cor1.2} refines the previously mentioned
estimate \eqref{LZdecay}.
\smallskip

Let us end this section with some notations we shall use throughout this paper.

\noindent{\bf Notations:}
We shall always denote $C$ to be a uniform constant
which may vary from line to line.
$a\lesssim b$ means that $a\leq Cb$, and $a\thicksim b$
means that both $a\leq Cb$ and $b\leq Ca$ hold.
For a Banach space $B$, we shall use the shorthand $L^p_T B$ for $\bigl\|\|\cdot\|_B\bigr\|_{L^p(0,T;dt)}$.
And we shall use $\widetilde{\nabla}=e_r\pa_r+e_z\pa_z$
and $\widetilde{u}=u^r e_r+u^z e_z $.

\section{Some Technical Lemmas}

\subsection{Axisymmetric Biot-Savart law}
The Biot-Savart law asserts that any divergence-free velocity field $u$
can be uniquely determined by its vorticity $\curl u$. Moreover, there holds
\begin{equation}\label{Biot}
\|\nabla u\|_{L^p}\leq C\frac{p^2}{p-1}\|\curl u\|_{L^p},\quad\forall\
1<p<\infty.
\end{equation}

For the special case when the velocity field is axisymmetric without swirl,
we can use $\omega^\theta$ to represent $\widetilde{u}=(u^r,u^z)$,
and the following estimates hold:

\begin{lem}[Proposition $2.4$ in \cite{GS}]\label{lemuruz}
{\sl Let $\alpha,~\beta\in[0,2]$
and $p,~q\in]1,\infty[$ satisfy
$$0\leq\beta-\alpha<1,\quad
\text{and}\quad \frac 1q=\frac 1p-\frac {1+\alpha-\beta}2.$$
If $r^\beta \ot\in L^p(\Omega)$, then $r^\alpha \widetilde{u}\in L^q(\Omega)$,
and there holds
$$\|r^\alpha \widetilde{u}\|_{L^q(\Omega)}\leq C\|r^\beta \ot\|_{L^p(\Omega)}.$$
}\end{lem}

\begin{lem}\label{lemur}
{\sl Let $\wt{u}=u^r e_r+u^z e_z$ and $\ot=\curl\wt u$.
Then for any $1<p<\infty$, there holds
$$\|\wt\nabla\ur\|_{L^p}+\|\wt\nabla\uz\|_{L^p}+
\bigl\|\frac{\ur}r\bigr\|_{L^p}\lesssim\|\ot\|_{L^p}.$$
}\end{lem}

\begin{proof}
Noting that $\ur$ is independent of $\theta$, we have
$$\bigl\|\frac{\ur}r\bigr\|_{L^p}=\bigl\|\frac{\pa_\theta}r(\ur e_r)\bigr\|_{L^p}
\leq\|\nabla\wt u\|_{L^p}.$$
Then this lemma follows directly from the Biot-Savart law \eqref{Biot} since
$\dive\wt u=0$.
\end{proof}

\subsection{The linearized system}

Next we investigate the solution operator $S(t)$ to the linearized system of \eqref{otut}, namely
$\omega^\theta(t)=S(t)\omega_0$ verifies
\begin{equation}\label{linear}
\left\{
\begin{split}
& \pa_t\ot-\bigl(\pa_r^2+\pa_z^2+\frac 1r \pa_r-\frac{1}{r^2}\bigr)\ot=0,   \quad(t,r,z)\in\R^+\times\Omega\\
& \ot|_{r=0}=0,\quad \ot|_{t=0} =\ot_0.
\end{split}
\right.
\end{equation}

\begin{lem}[Lemmas $3.1$ and $3.2$ in \cite{GS}]
{\sl For any $t>0$, one has
\begin{equation}\label{lemexpresion1}
\bigl(S(t)\omega_0\bigr)(r,z)=\frac{1}{4\pi t}\int_{\Omega}\frac{\bar{r}^{1/2}}{r^{1/2}}
{H}\Bigl(\frac{t}{r\bar{r}}\Bigr)\exp\Bigl(-\frac{(r-\bar{r})^2+(z-\bar{z})^2}{4t}\Bigr)
\omega_0(\bar{r},\bar{z})d\bar{r}d\bar{z},
\end{equation}
where the function ${H}:\R^+\rightarrow\R$ is defined by
$$
{H}(t)=\frac{1}{\sqrt{\pi t}}\int_{-\pi/2}^{\pi/2}e^{-\frac{\sin^2 \phi}{t}}\cos(2\phi) d\phi,\quad\forall\ t>0.$$
The function $H(t)$ is smooth on $\R^+$, and have the asymptotic expansions:
\begin{itemize}
\item[(i)] ${H}(t)=\frac{\pi^{1/2}}{4t^{3/2}}+\mathcal{O}\bigl(\frac{1}{t^{5/2}}\bigr),\
{H}'(t)=-\frac{3\pi^{1/2}}{8t^{5/2}}+\mathcal{O}\bigl(\frac{1}{t^{7/2}}\bigr),\ \text{as $t\rightarrow\infty$;}$

\item[(ii)] ${H}(t)=1-\frac{3t}{4}+\mathcal{O}(t^2),\
{H}'(t)=-\frac{3}{4}+\mathcal{O}(t),\ \text{as $t\rightarrow0$.}$
\end{itemize}
}\end{lem}

\begin{cor}\label{corH}
$t^\alpha H(t)$ and $t^\beta H'(t)$
are bounded on $\R^+$ provided
$0\leq\alpha\leq\frac 32$ and $0\leq\beta\leq\frac 52$.
\end{cor}

Now we can use this bound to derive the following weighted estimate:

\begin{prop}\label{lemsemi}
{\sl For any $1\leq p\leq q\leq \infty$ and any $f(r,z)\in L^p(\Omega)$,
the semigroup $\bigl(S(t)\bigr)_{t\geq0}$ is bounded in $L^p(\Omega)$,
and satisfy the following estimates.
\begin{itemize}
\item[(i)] For any $\alpha,~\beta$ satisfying $\alpha+\beta\leq0,
~\alpha\geq -1,~\beta\geq -1$, there holds:
\begin{equation}\label{semi1}
\bigl\|r^\alpha S(t)\widetilde{\nabla} (r^{\beta}f)\bigr\|_{L^q(\Omega)}
\leq\f{C}{t^{\f12-\f{\alpha+\beta}{2}+\f1p-\f1q}}\|f\|_{L^p(\Omega)},
\end{equation}
where $\widetilde{\nabla}=(\pa_r,\pa_z)$.
In particular, taking $\alpha=\beta=0$, we have:
\begin{equation}\label{semi2}
\|S(t)\widetilde{\nabla} f\|_{L^q(\Omega)}
\leq\frac{C}{t^{\frac 12+\frac 1p-\frac 1q}}\|f\|_{L^p(\Omega)}.
\end{equation}
\item[(ii)] For any $\alpha,~\beta$ satisfying $\alpha+\beta\leq1,
~\alpha\geq -1,~\beta\geq -1$, there holds:
\begin{equation}\label{semi3}
\|r^\alpha S(t) (r^{\beta-1} f)\|_{L^q(\Omega)}
\leq\f{C}{t^{\f12-\f{\alpha+\beta}{2}+\f1p-\f1q}}\|f\|_{L^p(\Omega)}.
\end{equation}
In particular, taking $\alpha=0,~\beta=1$, or $\alpha=\beta=0$, we have:
\begin{equation}\label{semi4}
\|S(t)f\|_{L^q(\Omega)}
\leq\frac{C}{t^{\frac 1p-\frac 1q}}\|f\|_{L^p(\Omega)},\quad
\|S(t) \bigl(\frac{f}{r}\bigr)\|_{L^q(\Omega)}
\leq\frac{C}{t^{\frac12+\frac 1p-\frac 1q}}\|f\|_{L^p(\Omega)}.
\end{equation}
\item[(iii)]  For any $\gamma,~\epsilon$ satisfying $\gamma+\epsilon\leq 0,
~\gamma\geq 0,~\epsilon\geq-1$, there holds:
\begin{equation}\label{semi5}
\bigl\|r^\gamma \widetilde{\nabla} S(t)(r^{\epsilon}f)\bigr\|_{L^q(\Omega)}
\leq\f{C}{t^{\f12-\f{\gamma+\epsilon}{2}+\f1p-\f1q}}\|f\|_{L^p(\Omega)}.
\end{equation}
\end{itemize}}
\end{prop}

\begin{proof}
For the proof of $({\rm i}),~({\rm ii})$, we first write
by using the expression \eqref{lemexpresion1} that
$$
r^\alpha\bigl( S(t)(r^{\beta-1}f)\bigr)(r,z)=\frac{1}{4\pi t}
\int_\Omega \frac{1}{r^{\frac 12-\alpha}\bar{r}^{\frac 12-\beta}}H\bigl(\frac{t}{r\bar{r}}\bigr)
\exp\Bigl(-\frac{(r-\bar{r})^2+(z-\bar{z})^2}{4t}\Bigr)
\cdot f(\bar{r},\bar{z})\,d\bar{r}\,d\bar{z},
$$
and by using integration by parts that
$$
r^\alpha\bigl( S(t)\widetilde{\nabla}(r^{\beta}f)\bigr)(r,z)=\frac{1}{4\pi t}
\int_\Omega \frac{\bar{r}^{\frac 12+\beta}}{r^{\frac 12-\alpha}}\exp\Bigl(-\frac{(r-\bar{r})^2+(z-\bar{z})^2}{4t}\Bigr)
\cdot\bigl(A_r,~A_z\bigr)\cdot f(\bar{r},\bar{z})\,d\bar{r}\,d\bar{z},
$$
where
$$A_r(\bar{r},\bar{z})=\frac{t}{r\bar{r}^2}{H}'\bigl(\frac{t}{r\bar{r}}\bigr)
-\bigl(\frac{1}{2\bar{r}}+\frac{r-\bar{r}}{2t}\bigr) H\bigl(\frac{t}{r\bar{r}}\bigr),\quad
A_z(\bar{r},\bar{z})=-\frac{z-\bar{z}}{2t}{H}\bigl(\frac{t}{r\bar{r}}\bigr).$$
Then $({\rm i}),~({\rm ii})$ can be proved by using Young's inequality
and the following two facts:

\textbf{Assertion a:} for any $\alpha,~\beta$ satisfying $\alpha+\beta\leq1$, $\alpha\geq -1$ and $\beta\geq -1$, there holds:
\begin{equation}\begin{split}\label{semi6}
\cA&\eqdef\Bigl(\frac{t}{r^{\frac 32-\alpha}\bar{r}^{\frac 32-\beta}}
\bigl|H'\bigl(\frac{t}{r\bar{r}}\bigr)\bigr|
+\frac{1}{r^{\frac 12-\alpha}\bar{r}^{\frac 12-\beta}}
\bigl|H\bigl(\frac{t}{r\bar{r}}\bigr)\bigr|\Bigr)
\cdot\exp\Bigl(-\frac{(r-\bar{r})^2+(z-\bar{z})^2}{4t}\Bigr)\\
&\lesssim \frac{1}{t^{\frac12-\frac{\alpha+\beta}{2}}}
\cdot\exp\Bigl(-\frac{(r-\bar{r})^2+(z-\bar{z})^2}{5t}\Bigr).
\end{split}\end{equation}

\textbf{Assertion b:} for any $\alpha,~\beta$ satisfying $\alpha+\beta\leq0$, $\alpha\geq -1$ and $\beta\geq -1$, there holds:
\begin{equation}\begin{split}\label{semi7}
\cB&\eqdef\frac{\bar{r}^{\frac 12+\beta}}{r^{\frac 12-\alpha}}
\Bigl(\bigl|\frac{r-\bar{r}}{2t}H\bigl(\frac{t}{r\bar{r}}\bigr)\bigr|
+\bigl|\frac{z-\bar{z}}{2t}H\bigl(\frac{t}{r\bar{r}}\bigr)\bigr|\Bigr)
\cdot\exp\Bigl(-\frac{(r-\bar{r})^2+(z-\bar{z})^2}{4t}\Bigr)\\
&\lesssim \frac{1}{t^{\frac12-\frac{\alpha+\beta}{2}}}
\cdot\exp\Bigl(-\frac{(r-\bar{r})^2+(z-\bar{z})^2}{5t}\Bigr).
\end{split}\end{equation}

\noindent$\bullet$ \textbf{The proof of Assertion a.}
The condition in Assertion a guarantees that
$0\leq\frac12-\frac{\alpha+\beta}{2}\leq\frac32$.
Let us first consider the case when $\alpha\geq\beta$ and
$\bar{r}\geq \frac r2$, we can deduce from Corollary \ref{corH} that
\begin{align*}
\cA
&\lesssim \Bigl(\frac{t}{r^{\frac 32-\alpha}\bar{r}^{\frac 32-\beta}}\bigl|\frac{r\bar{r}}{t}\bigr|^{\frac32-\frac{\alpha+\beta}{2}}
+\frac{1}{r^{\frac 12-\alpha}\bar{r}^{\frac 12-\beta}}\bigl|\frac{r\bar{r}}{t}\bigr|^{\frac12-\frac{\alpha+\beta}{2}}\Bigr)
\cdot\exp\Bigl(-\frac{(r-\bar{r})^2+(z-\bar{z})^2}{4t}\Bigr)\\
&\leq2^{1+\f{\alpha-\beta}2}\cdot \f{1}{t^{\frac12-\frac{\alpha+\beta}{2}}}\cdot
\exp\Bigl(-\frac{(r-\bar{r})^2+(z-\bar{z})^2}{4t}\Bigr).
\end{align*}
And when $\alpha\geq\beta$ and $\bar{r}<\frac r2$,
there holds $r<2|\bar{r}-r|$, then we can deduce from Corollary \ref{corH}
and the elementary inequality $e^{-\f{x}4}\lesssim x^{-s}e^{-\f{x}5}$,
which holds for any $x>0,~s>0$, that
\begin{align*}
\cA&\lesssim \Bigl(\frac{t}{r^{\frac 32-\alpha}\bar{r}^{\frac 32-\beta}}\bigl|\frac{r\bar{r}}{t}\bigr|^{\frac32-\beta}
+\frac{1}{r^{\frac 12-\alpha}\bar{r}^{\frac 12-\beta}}\bigl|\frac{r\bar{r}}{t}\bigr|^{\frac12-\beta}\Bigr)
\exp\Bigl(-\frac{(r-\bar{r})^2+(z-\bar{z})^2}{4t}\Bigr)\\
&\lesssim \frac{{r}^{\alpha-\beta}}{t^{\frac12-\beta}}\cdot
\bigl(\frac{t}{(r-\bar{r})^2+(z-\bar{z})^2}\bigr)^{\frac{\alpha-\beta}2}
\cdot\exp\Bigl(-\frac{(r-\bar{r})^2+(z-\bar{z})^2}{5t}\Bigr)\\
&\leq2^{\alpha-\beta}\cdot\f{1}{t^{\frac12-\frac{\alpha+\beta}{2}}}
\cdot\exp\Bigl(-\frac{(r-\bar{r})^2+(z-\bar{z})^2}{5t}\Bigr).
\end{align*}
On the other hand, when $\alpha\leq \beta$ and $\bar{r}\leq 2r$, we have
\begin{align*}
\cA
&\lesssim \Bigl(\frac{t}{r^{\frac 32-\alpha}\bar{r}^{\frac 32-\beta}}\bigl|\frac{r\bar{r}}{t}\bigr|^{\frac32-\frac{\alpha+\beta}{2}}
+\frac{1}{r^{\frac 12-\alpha}\bar{r}^{\frac 12-\beta}}\bigl|\frac{r\bar{r}}{t}\bigr|^{\frac12-\frac{\alpha+\beta}{2}}\Bigr)
\cdot\exp\Bigl(-\frac{(r-\bar{r})^2+(z-\bar{z})^2}{4t}\Bigr)\\
&\leq2^{1+\f{\beta-\alpha}2}\cdot \f{1}{t^{\frac12-\frac{\alpha+\beta}{2}}}
\cdot\exp\Bigl(-\frac{(r-\bar{r})^2+(z-\bar{z})^2}{4t}\Bigr).
\end{align*}
And when $\alpha\leq \beta$ and $\bar{r}>2r$,
there holds $\bar{r}<2|\bar{r}-r|$, hence
\begin{align*}
\cA
&\lesssim \Bigl(\frac{t}{r^{\frac 32-\alpha}\bar{r}^{\frac 32-\beta}}\bigl|\frac{r\bar{r}}{t}\bigr|^{\frac32-\alpha}
+\frac{1}{r^{\frac 12-\alpha}\bar{r}^{\frac 12-\beta}}\bigl|\frac{r\bar{r}}{t}\bigr|^{\frac12-\alpha}\Bigr)
\exp\Bigl(-\frac{(r-\bar{r})^2+(z-\bar{z})^2}{4t}\Bigr)\\
&\lesssim \frac{\bar{r}^{\beta-\alpha}}{t^{\frac12-\alpha}}\cdot
\bigl(\f{t}{(r-\bar{r})^2+(z-\bar{z})^2}\bigr)^{\frac{\beta-\alpha}2}
\cdot\exp\Bigl(-\frac{(r-\bar{r})^2+(z-\bar{z})^2}{5t}\Bigr)\\
&\leq2^{\beta-\alpha}\cdot\f{1}{t^{\frac12-\frac{\alpha+\beta}{2}}}
\cdot\exp\Bigl(-\frac{(r-\bar{r})^2+(z-\bar{z})^2}{5t}\Bigr).
\end{align*}
Combining the above four estimates completes the proof of Assertion a.

\noindent$\bullet$ \textbf{The proof of Assertion b.}
When $\frac r2\leq \bar{r}\leq 2r$, we have
\begin{align*}
\cB
&\lesssim \frac{\bar{r}^{\frac 12+\beta}}{r^{\frac 12-\alpha}}
\frac{|r-\bar{r}|+|z-\bar{z}|}{t}\bigl|\frac{r\bar{r}}{t}\bigr|^{-\frac{\alpha+\beta}{2}}
\bigl(\f{t}{(r-\bar{r})^2+(z-\bar{z})^2}\bigr)^{\frac12}
\exp\Bigl(-\frac{(r-\bar{r})^2+(z-\bar{z})^2}{5t}\Bigr)\\
&\lesssim \f{1}{t^{\frac12-\frac{\alpha+\beta}{2}}}
\cdot\exp\Bigl(-\frac{(r-\bar{r})^2+(z-\bar{z})^2}{5t}\Bigr).
\end{align*}
And when $\bar{r}>2r$ or $\bar{r}\leq \frac r2$,
there holds $\bar{r}+r<3|\bar{r}-r|$. The condition in Assertion b guarantees that
there exists a positive constant $\eta$ so that
$\max\bigl\{0,\frac12-\alpha,-\frac12-\beta,-\frac{1+\alpha+\beta}{2}\bigr\}
\leq \eta\leq\frac32$. Then we can deduce from Corollary \ref{corH} that
\begin{align*}
\cB
&\lesssim \frac{\bar{r}^{\frac 12+\beta}}{r^{\frac 12-\alpha}}\frac{|r-\bar{r}|+|z-\bar{z}|}{t}\bigl|\frac{r\bar{r}}{t}\bigr|^{\eta}
\bigl(\f{t}{(r-\bar{r})^2+(z-\bar{z})^2}\bigr)^{\frac{1+2\eta+\alpha+\beta}2}
\exp\Bigl(-\frac{(r-\bar{r})^2+(z-\bar{z})^2}{5t}\Bigr)\\
&\lesssim \frac{1}{t^{\frac12-\frac{\alpha+\beta}{2}}}
\exp\Bigl(-\frac{(r-\bar{r})^2+(z-\bar{z})^2}{5t}\Bigr).
\end{align*}
Combining the above two estimates completes the proof of Assertion b.
\smallskip

Now let us turn to the proof of $({\rm iii})$. We first write the
following explicit formula:
\begin{equation}\begin{split}\label{semi8}
r^\gamma\bigl(\pa_r,~\pa_z\bigr) \bigl(S(t)(r^\epsilon & f)\bigr)=\frac{1}{4\pi t}
\int_\Omega \frac{\bar{r}^{\frac 12+\epsilon}}{r^{\frac 12-\gamma}}\exp\Bigl(-\frac{(r-\bar{r})^2+(z-\bar{z})^2}{4t}\Bigr)
\cdot f(\bar{r},\bar{z})\\
&\cdot\Bigl(-\frac{t}{r^2\bar{r}}{H}'\bigl(\frac{t}{r\bar{r}}\bigr)
-\bigl(\frac{1}{2r}+\frac{r-\bar{r}}{2t}\bigr) H\bigl(\frac{t}{r\bar{r}}\bigr),
~-\frac{z-\bar{z}}{2t}{H}\bigl(\frac{t}{r\bar{r}}\bigr)\Bigr)\,d\bar{r}\,d\bar{z}.
\end{split}\end{equation}
Let us denote
$$B_1(r,z,\bar{r},\bar{z})\eqdef
\Bigl(\frac{t}{r^{\frac52-\gamma}\bar{r}^{\frac12-\epsilon}}\bigl|{H}'\bigl(\frac{t}{r\bar{r}}\bigr)\bigr|
+\frac{\bar{r}^{\frac12+\epsilon}}{r^{\frac32-\gamma}}\bigl|H\bigl(\frac{t}{r\bar{r}}\bigr)\bigr|\Bigr)
\exp\Bigl(-\frac{(r-\bar{r})^2+(z-\bar{z})^2}{4t}\Bigr),$$
$$B_2(r,z,\bar{r},\bar{z})\eqdef
\frac{\bar{r}^{\frac 12+\epsilon}}{r^{\frac 12-\gamma}}
\Bigl(\bigl|\frac{r-\bar{r}}{2t} H\bigl(\frac{t}{r\bar{r}}\bigr)\bigr|
+\bigl|\frac{z-\bar{z}}{2t}{H}\bigl(\frac{t}{r\bar{r}}\bigr)\bigr|\Bigr)
\exp\Bigl(-\frac{(r-\bar{r})^2+(z-\bar{z})^2}{4t}\Bigr).$$
$B_2$ can be bounded directly by using \eqref{semi7} with
$\alpha=\gamma,~\beta=\epsilon$ that
\begin{equation}\label{semi9}
B_2(r,z,\bar{r},\bar{z})\lesssim \frac{1}{t^{\frac12-\frac{\gamma+\epsilon}{2}}}
\cdot\exp\Bigl(-\frac{(r-\bar{r})^2+(z-\bar{z})^2}{5t}\Bigr).
\end{equation}

For $B_1$, when $\bar{r}\leq 2r$, we deduce from Corollary \ref{corH} that
\begin{align*}
B_1
&\lesssim \Bigl(\frac{t}{r^{\frac 52-\gamma}\bar{r}^{\frac 12-\epsilon}} \bigl|\frac{r\bar{r}}{t}\bigr|
^{\frac32-\frac{\gamma+\epsilon}2}
+\frac{\bar{r}^{\frac 12+\epsilon}}{r^{\frac 32-\gamma}} \bigl|\frac{r\bar{r}}{t}\bigr|
^{\frac12-\frac{\gamma+\epsilon}2}\Bigr)\cdot
\exp\Bigl(-\frac{(r-\bar{r})^2+(z-\bar{z})^2}{4t}\Bigr)\\
&\lesssim \frac{1}{t^{\frac12-\frac{\gamma+\epsilon}2}}\cdot\exp\Bigl(-\frac{(r-\bar{r})^2+(z-\bar{z})^2}{4t}\Bigr).
\end{align*}
When $\bar{r}>2r$, there holds $\bar{r}<2|\bar{r}-r|$, and thus we can get
\begin{align*}
B_1
&\lesssim \Bigl(\frac{t}{r^{\frac 52-\gamma}\bar{r}^{\frac 12-\epsilon}}\bigl|\frac{r\bar{r}}{t}\bigr|^{\frac52-\gamma}
+\frac{\bar{r}^{\frac 12+\epsilon}}{r^{\frac 32-\gamma}}\bigl|\frac{r\bar{r}}{t}\bigr|^{\frac32-\gamma}\Bigr)
\cdot\exp\Bigl(-\frac{(r-\bar{r})^2+(z-\bar{z})^2}{4t}\Bigr)\\
&\lesssim \frac{\bar{r}^{2-\gamma+\epsilon}}{t^{\frac32-\gamma}}\cdot
\bigl(\frac{t}{(r-\bar{r})^2+(z-\bar{z})^2}\bigr)^{1-\frac{\gamma-\epsilon}{2}}\cdot\exp\Bigl(-\frac{(r-\bar{r})^2+(z-\bar{z})^2}{5t}\Bigr)\\
&\lesssim \frac{1}{t^{\frac12-\frac{\gamma+\epsilon}2}}\cdot\exp\Bigl(-\frac{(r-\bar{r})^2+(z-\bar{z})^2}{5t}\Bigr).
\end{align*}
The above two estimates guarantee that we always have
\begin{equation}\label{semi10}
B_1(r,z,\bar{r},\bar{z})\lesssim \frac{1}{t^{\frac12-\frac{\gamma+\epsilon}2}}
\cdot\exp\Bigl(-\frac{(r-\bar{r})^2+(z-\bar{z})^2}{5t}\Bigr).
\end{equation}
Then \eqref{semi5} follows from \eqref{semi8}-\eqref{semi10}
and Young's inequality.
\end{proof}

\subsection{Some basic facts concerning axisymmetric N-S}

Our proof of Therem \ref{thm1.1} deeply relies on the strategy to decompose
the solution into two parts:
one is sufficiently small in the scaling invariant spaces,
while the other lies in some supercritical spaces.

Thus it is natural to study the axisymmetric
N-S with small initial data.
We shall prove that in this case, there would exist a global mild solution
which keeps sufficiently small in some well-designed space.
To state this precisely, let us introduce
\begin{equation}\begin{split}\label{defxt}
\|(&\ot,\ut)\|_{X_T}\eqdef
\sup_{0<t< T}\Bigl\{\|\ot(t)\|_{L^1(\Omega)\bigcap L^{\frac32}}
+t\|\ot(t)\|_{L^\infty}+\|\ut(t)\|_{L^2(\Omega)}\\
&+t^{\frac12}\bigl(\|\ut(t)\|_{L^{\infty}}
+\|r^{-1}\ut(t)\|_{L^2(\Omega)}+\|\widetilde{\nabla}\ut(t)\|_{L^2(\Omega)}\bigr)
+t^{\frac34}\|r^{-1}\ut(t)\|_{L^4(\Omega)}\Bigr\}\\
&+\|\ot\|_{L^2_T L^2(\Omega)\bigcap L_T^4L^2}+\|\ut\|_{L^4_T L^{4}(\Omega)}
+\|r^{-\frac35}\ut\|_{L^{\frac52}_T L^{\frac52}(\Omega)}
+\bigl\|r^{\frac23}\ot\bigr\|_{L^3_T L^3(\Omega)}\\
&+\bigl\|t^{\frac13}\ot\bigr\|_{L^3_T L^3(\Omega)}
+\bigl\|t^{\frac12}r^{-1}\ot\bigr\|_{L^2_T L^2(\Omega)}
+\bigl\|t^{\frac16}r^{\frac16}\widetilde{\nabla}\ot\bigr\|_{L^2_T L^2}
+\bigl\|t^{\frac16}\widetilde{\nabla}\ut\bigr\|_{L^{\frac{12}5}_T L^{\frac{12}5}(\Omega)}.
\end{split}\end{equation}
Then we can use a fixed-point argument to prove the following well-posedness
result in $X_T$. This proof is quite complicated, so we decide to put it
in Appendix \ref{appfix}.
\begin{prop}\label{globalsmall}
{\sl For any initial data $\ot_0\in L^1(\Omega)\cap L^{\frac32}$
and $\ut_0\in L^2(\Omega)$
with $\|\ot_0\|_{L^1(\Omega)\cap L^{\frac32}}$ and
$\|\ut_0\|_{L^2(\Omega)}$ sufficiently small,
the system \eqref{otut} would have a unique global mild solution $(\ot,\ut)$ in $X_\infty$.
Moreover, there exists some universal constant $C_1$ such that
\begin{equation}
\|(\ot,\ut)\|_{X_\infty}\leq 2C_1\bigl(\|\ot_0\|_{L^{\frac32}}+\|\ot_0\|_{L^1(\Omega)}
+\|\ut_0\|_{L^2(\Omega)}\bigr).
\end{equation}
}
\end{prop}

\begin{rmk}
The space $X_T$ consists of weighted norms of the form
$\|t^\alpha r^\beta\Box\|_{L_T^p L^q}$, where $\Box$ can be either $\ut,~\ot$
or their derivatives. These norms are chosen in considerations of:
$({\rm i})$ We should take $p\geq q$ so that we can apply Minkowski's inequality.
$({\rm ii})$ We should take $\alpha$ as small as possible,
so that the potential growth in time is still under control.
$({\rm iii})$ All these norms must be scaling invariant.
In this sense, the choice of $X_T$ seems more natural and sharp.
\end{rmk}

Let us end this section with another two useful inequalities.

\begin{lem}\label{lemrut}
{\sl For any regular enough solution of \eqref{1.2} on $[0,T]$, there holds
$$\|r\ut(t)\|_{L^p}\leq\|r\ut_0\|_{L^p}\quad
\forall \ p\in[2,\infty],~t\in [0,T].$$
}\end{lem}
\begin{proof}
Noting that $r\ut$ satisfies the following equation
$$\pa_t(r\ut)+\wt u\cdot\wt\nabla(r\ut)-(\pa_r^2+\pa_z^2+\frac 1r\pa_r)
(r\ut)+\f2r\pa_r(r\ut)=0.$$
Then when $p<\infty$, the $L^p$ estimate directly gives
$$\|r\ut(t)\|_{L^p}+\f{4(p-1)}{p}\bigl\|\wt\nabla|r\ut|^{\f p2}
\bigr\|_{L^2_t(L^2)}^2\leq\|r\ut_0\|_{L^p}.$$
And the case for $p=\infty$ follows by taking limit $p\rightarrow\infty$
in the above estimate.
\end{proof}

\begin{lem}[A special case of Theorem $2.1$ in \cite{BT02}]\label{lemBT02}
{\sl For any regular enough vector field $u$, the following
Sobolev-Hardy type inequality holds:
$$\bigl\|r^{-\f14}u\bigr\|_{L^4}\lesssim\|\nabla u\|_{L^2}.$$
}\end{lem}

\section{The Proof of the uniform in time estimate}\label{secprop2}

The aim of this section is to show that under the assumptions of Theorem \ref{thm1.1},
all the global strong solutions to \eqref{otut}
must satisfy the following uniform in time estimate:
\begin{equation}\begin{split}\label{thm1.1.2}
\ot\in L^\infty\bigl(\R^+;L^1(\Omega)\cap L^{\frac32}\bigr)
\cap L^4\bigl(\R^+;L^2\bigr),\quad
\ut\in L^\infty\bigl(\R^+;L^2(\Omega)\bigr).
\end{split}\end{equation}
This is the first step to prove Theorem \ref{thm1.1}.
And we shall split its proof into 3 steps: (i) First we prove
$\ot\in L^4_{\rm loc}\bigl(\R^+;L^2\bigr)$, which means that for any finite $T$,
$\|\ot\|_{L^4_T(L^2)}$ is finite. (ii) Then we prove that this estimate
is uniform in $T$, i.e. $\ot$ actually lies in $L^4\bigl(\R^+;L^2\bigr)$.
(iii) At last, with the $\ot\in L^4\bigl(\R^+;L^2\bigr)$ at hand,
we shall give the other estimates in \eqref{thm1.1.2}.

\subsection{The proof of $\ot\in L^4_{\rm loc}\bigl(\R^+;L^2\bigr)$}
Let us first focus on the short-time behavior.
For any given initial data $\ot_0\in L^{\frac32},~\ut_0\in L^2(\Omega)$,
the standard local well-posedness theory guarantees the existence of some
$T_1>0$, such that \eqref{otut} has a unique global solution on $[0,T_1]$ with
\begin{equation}\begin{split}\label{3.2}
&\|\ot\|_{L^\infty_{T_1}(L^{\f32})}^{\f32}
+\bigl\|\nabla|\ot|^{\f34}\bigr\|_{L^2_{T_1}(L^2)}^2
+\bigl\|r^{-1}|\ot|^{\f34}\bigr\|_{L^2_{T_1}(L^2)}^2
\leq C_0,\\
&\|\ut\|_{L^\infty_{T_1}(L^2(\Omega))}^2
+\|\nabla\ut\|_{L^2_{T_1}(L^2(\Omega))}^2
+\bigl\|r^{-1}\ut\bigr\|_{L^2_{T_1}(L^2(\Omega))}^2
\leq C_0.
\end{split}\end{equation}
where $C_0$ denotes some positive constant depending
only on the initial data, which may be different in each appearance.
In particular, this gives
\begin{equation}\label{3.3}
\|\ot\|_{L^4_{T_1}(L^2)}
\leq\|\ot\|_{L^\infty_{T_1}(L^{\f32})}^{\f58}
\|\ot\|_{L^{\f32}_{T_1}(L^{\f92})}^{\f38}
\leq\|\ot\|_{L^\infty_{T_1}(L^{\f32})}^{\f58}
\bigl\|\nabla|\ot|^{\f34}\bigr\|_{L^2_{T_1}(L^2)}^{\f12}\leq C_0.
\end{equation}

Then we only need to show that, for any finite $T_2>T_1$:
$\|\ot\|_{L^4([T_1,T_2];L^2)}$ is finite.
To do this, let us decompose $\bigl(\ot(0),\ut(0)\bigr)$ into
$\bigl(\ot_1(0),\ut_1(0)\bigr)+\bigl(\ot_2(0),\ut_2(0)\bigr)$ with
\begin{equation}\label{3.4}
\bigl(\ot_2(0),\,\ut_2(0)\bigr)=\bigl(\ot(0)
\cdot\chi_{\{r\geq A_1\}}(r),\,\ut(0)\cdot\chi_{\{r\geq A_1\}}(r)\bigr),
\end{equation}
where $\chi$ is the characteristic funtion, and $A_1$ is so
large that for some small positive constant $\ve$ to be chosen later,
and for the constants $C_1,~C_2$ appearing in \eqref{fixsmallcondi},
there holds
$$C_1\bigl(\|\ot_2(0)\|_{L^{\frac32}}+\|\ot_2(0)\|_{L^1(\Omega)}
+\|\ut_2(0)\|_{L^2(\Omega)}\bigr)< \min\bigl\{\f1{4C_2},\f{\ve}2\bigr\}.$$
Then by Proposition \ref{globalsmall}, we know that the following Cauchy problem
\begin{equation*}
\left\{\begin{split}
& \pa_t \ot_2+(u^r_2\pa_r+u^z_2\pa_z) \ot_2-(\pa_r^2+\pa_z^2+\frac 1r\pa_r-\frac{1}{r^2})\omega^\theta_2-\f{u^r_2 \omega^\theta_2}{r}
-\f{2u^\theta_2 \pa_z u^\theta_2}{r}=0,\\
& \pa_t \ut_2+(u^r_2\pa_r+u^z_2\pa_z)\ut_2
-(\pa_r^2+\pa_z^2+\frac 1r\pa_r-\f{1}{r^2})u^\theta_2+\f{u^r_2 u^\theta_2}{r}=0,\\
& (\ot_2,\ut_2)|_{t=0}=\bigl(\ot_2(0), \ut_2(0)\bigr)
\end{split}\right.
\end{equation*}
has a unique global solution $(\ot_2,\ut_2)$ in $X_\infty$,
where $\ur_2,\uz_2$ are induced by $\ot_2$ via the Biot-Savart law.
Moreover, this solution remains small all the time satisfying
\begin{equation}\label{spaceot2ut2}
\|(\ot_2,\ut_2)\|_{X_\infty}<\ve.
\end{equation}

Now let us temporarily go back to Euclidean coordinates.
We first get, by using \eqref{spaceot2ut2}, the Biot-Savart law
and Lemma \ref{lemrut} that for any $t>0$, there holds
\begin{equation}\label{u2L3}
\|u_2(t)\|_{L^3}
\leq C\|\ot_2(t)\|_{L^{\frac32}}+\|\ut_2(t)\|_{L^2(\Omega)}^{\f23}
\|r\ut_2(t)\|_{L^\infty}^{\f13}
\leq C\ve+\ve^{\f23}\|r\ut(0)\|_{L^\infty}^{\f13},
\end{equation}
which is sufficiently small. And the remainder $(\ot_1,\ut_1)=(\ot,\ut)-(\ot_2,\ut_2)$
satisfies
\begin{equation}\label{5.2}
\pa_t u_1+u\cdot\nabla u_1+u_1\cdot\nabla u_2
-\D u_1+\nabla P_1=0,\quad\dive u_1=0.
\end{equation}
In view of Lemma \ref{lemuruz} and the assumptions on $(\ot(0),\ut(0))$, we have
\begin{equation}\begin{split}\label{initialu1}
&\|\ut_1(0)\|_{L^2}=\bigl\|r^{\f12}\ut(0)\cdot\chi_{\{r< A_1\}}\bigr\|_{L^2(\Omega)}
\leq A_1^{\f12}\|\ut(0)\|_{L^2(\Omega)}<\infty,\\
\|\wt u_1(0)&\|_{L^2}\lesssim\bigl\|r^{\f12}\ot_1(0)\bigr\|_{L^{\f32}}
=\bigl\|r^{\f12}\ot(0)\cdot\chi_{\{r< A_1\}}\bigr\|_{L^{\f32}}
\leq A_1^{\f12}\|\ot(0)\|_{L^{\f32}}<\infty.
\end{split}\end{equation}
Hence we can apply $L^2$ energy estimates to \eqref{5.2} to get
$$\f12\f{d}{dt}\|u_1\|_{L^2}^2+\|\nabla u_1\|_{L^2}^2
=-\int_{\R^3}(u_1\cdot\nabla u_2)\cdot u_1\,dx.$$
In view of $\dive u_1=0$ and the bound \eqref{u2L3},
the right-hand side can be estimated as
$$-\int_{\R^3}(u_1\cdot\nabla u_2)\cdot u_1\,dx
=\int_{\R^3}(u_1\otimes u_2):\nabla u_1\,dx
\leq C\|u_2\|_{L^3}\|\nabla u_1\|_{L^2}^2
<\f12\|\nabla u_1\|_{L^2}^2.$$
As a result, we achieve
\begin{equation}\label{6.4}
\|u_1\|_{L^\infty_t(L^2)}^2+\|\nabla u_1\|_{L^2_t(L^2)}^2
\leq \|u_1(0)\|_{L^2}^2,\quad\forall\ t>0.
\end{equation}

Now let us focus on the time interval $[T_1, T_2]$.
Using \eqref{6.4}, we have
$$\|u_1\|_{L^4_t(L^3)}
\leq\|u_1\|_{L^\infty_t(L^2)}^{\f12}
\|\nabla u_1\|_{L^2_t(L^2)}^{\f12}
\leq C_0.$$
Then for any $t_0\in[T_1,T_2]$ and
$\mathbf{r}^2=\min\bigl\{1,\f12 T_1\bigr\}$,
H\"older's inequality gives
$$\|u_1\|_{L^3_{t,x}((t_0-\mathbf{r}^2,t_0]\times\R^3)}
\leq\|u_1\|_{L^4((t_0-\mathbf{r}^2,t_0];L^3(\R^3))}
\leq C_0.$$
On the other hand, in view of \eqref{u2L3}, the similar bound
also holds for $u_2$, namely
$$\|u_2\|_{L^3_{t,x}((t_0-\mathbf{r}^2,t_0]\times\R^3)}
\leq\|u_2\|_{L^\infty((t_0-\mathbf{r}^2,t_0];L^3)}
\leq C\ve+\ve^{\f23}\|r\ut(0)\|_{L^\infty}^{\f13}.$$
As a result, for $u=u_1+u_2$, there holds
$$\|u\|_{L^3_{t,x}((t_0-\mathbf{r}^2,t_0]\times\R^3)}
\leq C_0.$$
Then we can appropriately
choose some pressure $P=-\D^{-1}\dive\dive(u\otimes u)$ so that
\begin{align*}
\|P\|_{L^{\f32}_{t,x}((t_0-\mathbf{r}^2]\times\R^3)}
\lesssim\|u\|_{L^3_{t,x}((t_0-\mathbf{r}^2]\times\R^3)}^2\leq C_0.
\end{align*}
Thus for any small constant $\ve>0$, we can find some
compact set $K_1\subset\R^3$ such that
\begin{equation}\begin{split}\label{ep1}
\|u&\|_{L^3_{t,x}((t_0-\mathbf{r}^2,t_0]\times(\R^3\setminus K_1))}
+\|P\|_{L^{\f32}_{t,x}((t_0-\mathbf{r}^2]\times(\R^3\setminus K_1))}\\
&=\|u\|_{L^3(\R^3\setminus K_1;L^3(t_0-\mathbf{r}^2,t_0))}
+\|P\|_{L^{\f32}(\R^3\setminus K_1;L^{\f32}(t_0-\mathbf{r}^2,t_0))}
<\ve.
\end{split}\end{equation}
Then we find a little larger subset $K_2\subset\R^3$ such that
$$z_0=(t_0,x_0)\in[T_1,T_2]\times(\R^3\setminus K_2)\Rightarrow
Q_{\mathbf{r}}(z_0)\subset(t_0-\mathbf{r}^2,t_0]\times(\R^3\setminus K_1),$$
where $Q_{\mathbf{r}}(z_0)=B_r(x_0)\times(t_0-{\mathbf{r}}^2,t_0]$.
Then for such $z_0$, \eqref{ep1} implies
\begin{equation}\label{ep2}
\|u\|_{L^3_{t,x}(Q_{\mathbf{r}}(z_0))}
+\|P\|_{L^{\f32}_{t,x}(Q_{\mathbf{r}}(z_0))}
<\ve.
\end{equation}

Now in view of the estimate \eqref{ep2},
we can apply $\ve$-regularity criterion (see \cite{CKN, Lin})
to get some constant $C_0(T_1,T_2,K_2)$ which
depends on the initial data, $T_1,~T_2$ and $K_2$ such that
\begin{equation}\label{4.121}
\|u\|_{L^\infty_{t,x}([T_1,T_2]\times(\R^3\setminus K_2))}
+\|\nabla u\|_{L^\infty_{t,x}([T_1,T_2]\times(\R^3\setminus K_2))}
\leq C_0(T_1,T_2,K_2).
\end{equation}
While by assumption, $u$ is H\"older continuous at any point in $\R^+\times\R^3$.
Thus $|u|$ and $|\nabla u|$ are uniformly bounded
on the compact set $[T_1,T_2]\times K_2$,
which together with \eqref{4.121} implies
\begin{equation}\label{4.12}
\|u\|_{L^\infty_{t,x}([T_1,T_2]\times\R^3)}
+\|\nabla u\|_{L^\infty_{t,x}([T_1,T_2]\times\R^3)}
\leq C_0(T_1,T_2,K_2).
\end{equation}

Next, by applying $L^{\f32}$ energy estimate to $\ot$
in \eqref{otut}, we get
\begin{equation}\begin{split}\label{3.11}
&\f{d}{dt}\|\ot\|_{L^{\f32}}^{\f32}
+\bigl\|\nabla|\ot|^{\f34}\bigr\|_{L^2}^2
+\bigl\|r^{-1}|\ot|^{\f34}\bigr\|_{L^2}^2
\leq C\int_{\R^3}\bigl(\f{|u^r\ot|}{r}+\f{|\ut\pa_z\ut|}{r}\bigr)
\cdot|\ot|^{\f12}\,dx\\
&\leq C\|\ur\|_{L^\infty}\bigl\|r^{-1}|\ot|^{\f34}\bigr\|_{L^2}
\bigl\||\ot|^{\f34}\bigr\|_{L^2}
+C\|\ut\|_{L^2(\Omega)}\|\pa_z\ut\|_{L^2(\Omega)}
\|\ot\|_{L^\infty}^{\f12}\\
&\leq \f14\bigl(\bigl\|r^{-1}|\ot|^{\f34}\bigr\|_{L^2}^2
+\|\pa_z\ut\|_{L^2(\Omega)}^2\bigr)
+C\bigl(\|\ot\|_{L^{\f32}}^{\f32}
+\|\ut\|_{L^2(\Omega)}^2\bigr)\bigl(\|\ur\|_{L^\infty}^2
+\|\ot\|_{L^\infty}\bigr).
\end{split}\end{equation}
And by applying $L^2(\Omega)$ energy estimate to $\ut$
in \eqref{otut} that
\begin{align*}
\frac12\f{d}{dt}\|\ut\|_{L^2(\Omega)}^2
+\|\nabla\ut\|_{L^2(\Omega)}^2+\bigl\|\f{\ut}r\bigr\|_{L^2(\Omega)}^2
&=\int_{\Omega}\Bigl(\f{\pa_r\ut}{r}-\f{u^r u^\theta}{r}
-(u^r\pa_r+u^z\pa_z)\ut\Bigr)\cdot\ut\,drdz\\
&=\int_{\Omega}\Bigl(\f{\pa_r\ut}{r}-\f32\f{u^r u^\theta}{r}\Bigr)\cdot\ut\,drdz,
\end{align*}
where in the last step we have used integration by parts
and the divergence-free condition.
Furthermore, by using the fact that $\ut|_{r=\infty}=0$,
integration by parts gives
$$\int_{\Omega}\f{\pa_r\ut}{r}\cdot\ut\,drdz
=-\int_{\R}\f{|\ut|^2}r\Big|_{r=0}\,dz
+\f12\bigl\|\f{\ut}r\bigr\|_{L^2(\Omega)}^2
\leq\f12\bigl\|\f{\ut}r\bigr\|_{L^2(\Omega)}.$$
As a result, we achieve
\begin{equation}\begin{split}\label{3.12}
\f{d}{dt}\|\ut\|_{L^2(\Omega)}^2
+\|\nabla\ut\|_{L^2(\Omega)}^2+\bigl\|\f{\ut}r\bigr\|_{L^2(\Omega)}^2
&\leq C\Bigl|\int_{\Omega}\f{u^r u^\theta}{r}\cdot\ut\,drdz\Bigr|\\
&\leq C\|\ur\|_{L^\infty}\|\ut\|_{L^2(\Omega)}
\bigl\|\f{\ut}r\bigr\|_{L^2(\Omega)}^2\\
&\leq \f14\bigl\|\f{\ut}r\bigr\|_{L^2(\Omega)}^2
+C\|\ur\|_{L^\infty}^2\|\ut\|_{L^2(\Omega)}^2.
\end{split}\end{equation}
By putting the estimates \eqref{3.11} and \eqref{3.12} together,
and then applying Gronwall's inequality,
using the estimates \eqref{3.2} and \eqref{4.12}, we can obtain
\begin{align*}
&\|\ot(t)\|_{L^{\f32}}^{\f32}
+\|\ut(t)\|_{L^2(\Omega)}^2
+\int_{T_1}^t\Bigl(\bigl\|\nabla|\ot|^{\f34}\bigr\|_{L^2}^2
+\bigl\|\nabla \ut\bigr\|_{L^2(\Omega)}^2\Bigr)(t')\,dt'\\
&\leq\bigl(\|\ot(T_1)\|_{L^{\f32}}^{\f32}
+\|\ut(T_1)\|_{L^2(\Omega)}^2\bigr)
\exp\Bigl((t-T_1)\bigl(\|\ur\|_{L^\infty([T_1,t];L^\infty)}^2
+\|\ot\|_{L^\infty([T_1,t];L^\infty)}\bigr)\Bigr)\\
&\leq C_0(T_1,T_2,K_2),
\quad\forall\ t\in[T_1,T_2].
\end{align*}
In particular, this implies
\begin{equation}\label{3.14}
\|\ot\|_{L^4([T_1,T_2];L^2)}
\leq\|\ot\|_{L^\infty([T_1,T_2];L^{\f32})}^{\f58}
\bigl\|\nabla|\ot|^{\f34}\bigr\|_{L^2([T_1,T_2];L^2)}^{\f12}
\leq C_0(T_1,T_2,K_2).
\end{equation}

Now by combining the estimates \eqref{3.3} and \eqref{3.14} together, we deduce
$\ot\in L^4([0,T_2];L^2)$, which is exactly the desired
$\ot\in L^4_{\rm loc}\bigl(\R^+;L^2\bigr)$ as $T_2$ can be arbitrary.

\subsection{The proof of $\ot\in L^4\bigl(\R^+;L^2\bigr)$}
The energy inequality \eqref{6.4} implies $\|\nabla u_1(t)\|_{L^2}^2$
is small for almost every large time $t$,
which ensures the existence of some $t_1>0$ such that
\begin{equation}\label{6.5}
\|\ot_1(t_1)\|_{L^2}<A_1^{-\f12}\ve\andf
\bigl\|\f{\ut_1}{r}(t_1)\bigr\|_{L^2}<A_1^{-\f12}\ve,
\end{equation}
where $A_1$ is the cut-off point appearing in \eqref{3.4}.

In the following, we shall resolve $(\ot_1,\ut_1)$ from $t_1$, which satisfies
\begin{equation}\label{ot1ut1}
\left\{\begin{split}
& \pa_t \ot_1+\wt u\cdot\wt\nabla\ot_1+\wt u_1\cdot\wt\nabla\ot_2
-(\D-\f{1}{r^2})\omega^\theta_1
-\frac{u^r_1\ot+u^r_2 \ot_1}{r}
-\frac{2u^\theta_2 \pa_z u^\theta_1+2u^\theta_1 \pa_z\ut}{r}=0, \\
& \pa_t \ut_1+\wt u\cdot\wt\nabla\ut_1+\wt u_1\cdot\wt\nabla\ut_2
-(\D-\f{1}{r^2})\ut_1
+\frac{u^r_1 u^\theta_1+\ur_1\ut_2+\ur_2\ut_1}{r}=0.
\end{split}\right.
\end{equation}
For any $t>t_1$, let us introduce
$$E_1^2(t)\eqdef\|\ot_1\|_{L^\infty([t_1,t);L^2)}^2
+\|\nabla\ot_1\|_{L^2([t_1,t);L^2)}^2
+\bigl\|\f{\ot_1}{r}\bigr\|_{L^2([t_1,t);L^2)}^2,$$
and
$$E_2^2(t)\eqdef\bigl\|\f{\ut_1}r\bigr\|_{L^\infty([t_1,t);L^2)}^2
+\bigl\|\nabla\f{\ut_1}r\bigr\|_{L^2([t_1,t);L^2)}^2.$$

The $L^2$ estimate of $\ot_1$ in \eqref{ot1ut1} gives
\begin{equation}\begin{split}\label{6.6}
\f12\f{d}{dt}\|\ot_1\|_{L^2}^2+\|\nabla\ot_1\|_{L^2}^2
+\bigl\|\f{\ot_1}{r}&\bigr\|_{L^2}^2
=\int_{\R^3}\bigl(-\wt u_1\cdot\wt\nabla\ot_2+\f{\ur_1(\ot_1+\ot_2)}{r}
+\f{\ur_2\ot_1}{r}\bigr)\ot_1\,dx\\
&+\int_{\R^3}\f{2\ut_2\pa_z\ut_1+2\ut_1(\pa_z\ut_1+\pa_z\ut_2)}{r}\ot_1\,dx
\eqdef\cD_1+\cD_2.
\end{split}\end{equation}
By using H\"older's inequality, Sobolev's embedding theorem
and Lemma \ref{lemuruz}, we have
\begin{align*}
|\cD_1|&\leq\|\wt u_1\|_{L^3}\|\nabla\ot_2\|_{L^2}\|\ot_1\|_{L^6}
+\|\ur_1\|_{L^3}\bigl(\|\ot_1\|_{L^6}+\|\ot_2\|_{L^6}\bigr)
\bigl\|\f{\ot_1}{r}\bigr\|_{L^2}
+\|\ur_2\|_{L^3}\|\ot_1\|_{L^6}\bigl\|\f{\ot_1}{r}\bigr\|_{L^2}\\
&\lesssim\|\wt u_1\|_{L^2}^{\f12}\|\ot_1\|_{L^2}^{\f12}
\bigl(\|\nabla\ot_1\|_{L^2}+\|\nabla\ot_2\|_{L^2}\bigr)
\bigl(\|\nabla\ot_1\|_{L^2}+\bigl\|\f{\ot_1}{r}\bigr\|_{L^2}\bigr)
+\|\ot_2\|_{L^{\f32}}\|\nabla\ot_1\|_{L^2}\bigl\|\f{\ot_1}{r}\bigr\|_{L^2}.
\end{align*}
For the second term $\cD_2$, integration by parts gives
$$\int_{\R^3}\f{2\ut_2\pa_z\ut_1}{r}\ot_1\,dx
+\int_{\R^3}\f{2\ut_1\pa_z\ut_2}{r}\ot_1\,dx
=-\int_{\R^3}\f{2\ut_1\ut_2}{r}\pa_z\ot_1\,dx.$$
In view of this, we can obtain
\begin{align*}
|\cD_2|&\leq \|\ut_1\|_{L^3}
\bigl\|\pa_z\f{\ut_1}{r}\bigr\|_{L^2}\|\ot_1\|_{L^6}
+\|\ut_2\|_{L^3}\bigl\|\f{\ut_1}{r}\bigr\|_{L^6}\|\pa_z\ot_1\|_{L^2}\\
&\lesssim\bigl(\|\ut_1\|_{L^2}^{\f13}
\bigl\|\f{\ut_1}{r}\bigr\|_{L^2}^{\f13}\|r\ut_1\|_{L^\infty}^{\f13}
+\|\ut_2\|_{L^2(\Omega)}^{\f23}\|r\ut_2\|_{L^\infty}^{\f13}\bigr)
\bigl\|\nabla\f{\ut_1}{r}\bigr\|_{L^2}\|\nabla\ot_1\|_{L^2}.
\end{align*}
Now by substituting the above two estimates into \eqref{6.6},
integrating in time,
and then using the bounds \eqref{spaceot2ut2},~\eqref{6.4},~\eqref{6.5}
and Lemma \ref{lemrut}, we deduce for any $t>t_1$, there holds
\begin{equation}\begin{split}\label{6.7}
E_1^2\lesssim A_1^{-1}\ve^2+\|u_1(0)\|_{L^2}^{\f12}E_1^{\f32}(E_1+\ve)
+\ve E_1^2+\bigl(\|u_1(0)\|_{L^2}^{\f13}E_2^{\f13}
+\ve^{\f23}\bigr)\|r\ut(0)\|_{L^\infty}^{\f13}E_1E_2.
\end{split}\end{equation}

On the other hand, we can derive from \eqref{ot1ut1} that
${\ut_1}/{r}$ satisfies
\begin{equation}\label{eqtut1r}
\pa_t\bigl(\f{\ut_1}{r}\bigr)+\wt u\cdot\wt\nabla\bigl(\f{\ut_1}{r}\bigr)
+\wt u_1\cdot\wt\nabla\bigl(\f{\ut_2}{r}\bigr)
-\D\bigl(\f{\ut_1}{r}\bigr)
-\f2r\pa_r\bigl(\f{\ut_1}{r}\bigr)
+2\f{u^r u^\theta_1+\ur_1\ut_2}{r^2}=0.
\end{equation}
In view of the fact that $\ut_1|_{r=\infty}=0$, integration by parts gives
$$-\int_{\R^3}\f2r\pa_r\f{\ut_1}{r}\cdot\f{\ut_1}{r}\,dx
=-2\pi\int_{\Omega}\pa_r\bigl|\f{\ut_1}{r}\bigr|^2\,drdz
=2\pi\int_{\R}\bigl|\f{\ut_1}{r}\bigr|^2\,dz\Big|_{r=0}\geq0.$$
By using this, taking $L^2$ inner product of \eqref{eqtut1r}
with $\ut_1/r$ gives
\begin{equation}\begin{split}\label{6.8}
\f12\f{d}{dt}\bigl\|\f{\ut_1}{r}\bigr\|_{L^2}^2
+\bigl\|\nabla\f{\ut_1}{r}\bigr\|_{L^2}^2
\leq \int_{\R^3}\bigl(\wt u_1\otimes\f{\ut_2}{r}\bigr)
:\wt\nabla\f{\ut_1}{r}\,dx
-2\int_{\R^3}\f{\ur \ut_1+\ur_1\ut_2}{r^2}
\cdot\f{\ut_1}{r}\,dx.
\end{split}\end{equation}

For the terms on the right-hand side of \eqref{6.8}, we first get
\begin{align*}
\Bigl|\int_{\R^3}\f{\ur u^\theta_1}{r^2}\cdot\f{\ut_1}{r}\,dx\Bigr|
&\leq\bigl(\|\ur_1\|_{L^2}^{\f12}\bigl\|\f{\ur_1}{r}\bigr\|_{L^2}^{\f12}
+\bigl\|\f{\ur_2}{\sqrt r}\bigr\|_{L^2}\bigr)
\bigl\|\f{\ut_1}{r\cdot r^{\f14}}\bigr\|_{L^4}^2\\
&\leq\bigl(\|u_1(0)\|_{L^2}^{\f12}\|\ot_1\|_{L^2}^{\f12}
+\|\ot_2\|_{L^{\f32}}\bigr)
\bigl\|\nabla\f{\ut_1}{r}\bigr\|_{L^2}^2,
\end{align*}
where we have used \eqref{6.4},
Lemmas \ref{lemuruz},~\ref{lemur} and \ref{lemBT02} in the last step.
Similarly, we have
\begin{align*}
\int_{\R^3}\bigl|\f{\ur_1 u^\theta_2}{r^2}\cdot\f{\ut_1}{r}\bigr|
+\bigl|\bigl(\wt u_1\otimes\f{\ut_2}{r}\bigr)
:\wt\nabla\f{\ut_1}{r}\bigr|\,dx
&\leq\bigl\|\f{\ur_1}{r}\bigr\|_{L^{\f{30}{13}}}
\bigl\|\f{\ut_2}{r}\bigr\|_{L^{\f52}}\bigl\|\f{\ut_1}{r}\bigr\|_{L^6}
+\|\wt u_1\|_{L^{10}}
\bigl\|\f{\ut_2}{r}\bigr\|_{L^{\f52}}\bigl\|\wt\nabla\f{\ut_1}{r}\bigr\|_{L^2}\\
&\lesssim\|\ot_1\|_{L^{\f{30}{13}}}\bigl\|r^{-\f35}\ut_2\bigr\|_{L^{\f52}(\Omega)}
\bigl\|\nabla\f{\ut_1}{r}\bigr\|_{L^2}\\
&\lesssim\|\ot_1\|_{L^2}^{\f45}\|\nabla\ot_1\|_{L^2}^{\f15}
\bigl\|r^{-\f35}\ut_2\bigr\|_{L^{\f52}(\Omega)}
\bigl\|\nabla\f{\ut_1}{r}\bigr\|_{L^2}.
\end{align*}
Substituting the above estimates into \eqref{6.8},
and then integrating the resulting inequality in time,
we deduce for any $t>t_1$, there holds
\begin{equation}\begin{split}\label{6.9}
E_2^2\lesssim A_1^{-1}\ve^2+\|u_1(0)\|_{L^2}^{\f12}E_1^{\f12}E_2^2
+\ve E_1E_2+\ve E_2^2.
\end{split}\end{equation}

Now by combining the estimates \eqref{6.8} and \eqref{6.9} together,
as well as the bound \eqref{initialu1} for $\|u_1(0)\|_{L^2}$,
a standard continuity argument guarantees the existence of some
small positive constant $\ve_0$, which only depends on
$\|\ut(0)\|_{L^2(\Omega)},~
\|\ot(0)\|_{L^{\f32}}$ and $\|r\ut(0)\|_{L^\infty}$,
such that whenever $\ve<\ve_0$, then
there exists a universal constant $C_3$ such that
$$E_1^2(t)+E_2^2(t)\leq C_3 A_1^{-1}\ve^2\cdot\max\bigl\{1,\|\ut(0)\|_{L^2(\Omega)},
\|\ot(0)\|_{L^{\f32}},\|r\ut(0)\|_{L^\infty}\bigr\}.$$
In particular, this gives the following uniform in time estimate:
$$\|\ot_1(t)\|_{L^2}^2 \leq C_3 A_1^{-1}\ve^2\cdot\max\bigl\{1,\|\ut(0)\|_{L^2(\Omega)},
\|\ot(0)\|_{L^{\f32}},\|r\ut(0)\|_{L^\infty}\bigr\},\quad\forall\ t>t_1.$$
Combining this with the energy inequality \eqref{6.4} which says
$$\|\ot_1\|_{L^2([t_1,\infty);L^2)}^2\leq\|\nabla u_1\|_{L^2([t_1,\infty);L^2)}^2
\leq C\|u_1(0)\|_{L^2}^2,$$
the fact $\ot\in L^4_{{\rm loc}}(\R^+;L^2)$,
and the bound \eqref{spaceot2ut2} for $\ot_2$, we obtain
\begin{equation}\begin{split}\label{uniot}
\|\ot\|_{L^4(\R^+;L^2)}
&\leq\|\ot\|_{L^4([0,t_1);L^2)}+\|\ot_1\|_{L^4([t_1,\infty);L^2)}
+\|\ot_2\|_{L^4([t_1,\infty);L^2)}\\
&\leq\|\ot\|_{L^4([0,t_1);L^2)}
+\|\ot_1\|_{L^2([t_1,\infty);L^2)}^{\f12}\|\ot_1\|_{L^\infty([t_1,\infty);L^2)}^{\f12}
+\ve<\infty.
\end{split}\end{equation}

\subsection{The proof of \eqref{thm1.1.2}}
With the previously proved $\ot\in L^4\bigl(\R^+;L^2\bigr)$ at hand,
the remaining estimates in \eqref{thm1.1.2} can be proved by standard energy method.
So we just give a sketch.

We first get, by using the first line of \eqref{3.12} that
$$\f{d}{dt}\|\ut\|_{L^2(\Omega)}^2
+\|\nabla\ut\|_{L^2(\Omega)}^2+\bigl\|\f{\ut}r\bigr\|_{L^2(\Omega)}^2
\leq C\Bigl|\int_{\Omega}\f{u^r u^\theta}{r}\cdot\ut\,drdz\Bigr|.$$
Then Lemma \ref{lemuruz} and Sobolev embedding theorem on the upper half plane give
\begin{align*}
C\Bigl|\int_{\Omega}\f{u^r\ut}r\cdot\ut\,drdz\Bigr|
&\leq C\|\ur\|_{L^4(\Omega)}
\bigl\|\f{\ut}r\bigr\|_{L^2(\Omega)}\|\ut\|_{L^4(\Omega)}\\
&\leq C\bigl\|r^{\f12}\ot\bigr\|_{L^2(\Omega)}
\bigl\|\f{\ut}r\bigr\|_{L^2(\Omega)}
\|\ut\|_{L^2(\Omega)}^{\f12}\|\nabla\ut\|_{L^2(\Omega)}^{\f12}\\
&\leq \f14\bigl(\|\nabla\ut\|_{L^2(\Omega)}^2
+\bigl\|\f{\ut}r\bigr\|_{L^2(\Omega)}\bigr)
+C\|\ut\|_{L^2(\Omega)}^2\|\ot\|_{L^2}^4.
\end{align*}
As a result, we achieve
$$\f{d}{dt}\|\ut\|_{L^2(\Omega)}^2
+\|\nabla\ut\|_{L^2(\Omega)}^2+\bigl\|\f{\ut}r\bigr\|_{L^2(\Omega)}^2
\lesssim\|\ut\|_{L^2(\Omega)}^2\|\ot\|_{L^2}^4.$$
Then Gronwall's inequality together with the bound \eqref{uniot} implies
that for any $t>0$, there holds uniformly
\begin{equation}\label{ut2L2Omega3}
\|\ut(t)\|_{L^2(\Omega)}^2
+\|\nabla\ut\|_{L^2_tL^2(\Omega)}^2+\bigl\|\f{\ut}r\bigr\|_{L^2_tL^2(\Omega)}^2
\leq\|\ut(0)\|_{L^2(\Omega)}^2\exp\bigl(C\|\ot\|_{L^4(\R^+;L^2)}\bigr).
\end{equation}

Next, by applying $L^{\f32}$ energy estimate to $\ot$
in \eqref{otut} gives
\begin{align*}
\f{d}{dt}\|\ot&\|_{L^{\f32}}^{\f32}
+\bigl\|\nabla|\ot|^{\f34}\bigr\|_{L^2}^2
+\bigl\|r^{-1}|\ot|^{\f34}\bigr\|_{L^2}^2
\leq C\int_{\R^3}\bigl(\f{|\ut\pa_z\ut|}{r}+\f{|u^r\ot|}{r}\bigr)
\cdot|\ot|^{\f12}\,drdz\\
&\leq C\bigl\||\ot|^{\f12}\bigr\|_{L^3}
\bigl\|r^{-\f12}\pa_z\ut\bigr\|_{L^2}\bigl\|r^{-\f12}\ut\bigr\|_{L^6}
+C\bigl\||\ot|^{\f34}\bigr\|_{L^3}\bigl\||\ot|^{\f34}\bigr\|_{L^6}
\bigl\|\f{\ur}r\bigr\|_{L^2}\eqdef\cG_1+\cG_2.
\end{align*}
Noticing that $\ut$ is independent of $\theta$, we have
$$\bigl\|r^{-\f12}\ut\bigr\|_{L^6}
\lesssim\bigl\|\nabla\bigl(r^{-\f12}\ut\bigr)\bigr\|_{L^2}
=\bigl\|\wt\nabla\bigl(r^{-\f12}\ut\bigr)\bigr\|_{L^2}
\thicksim\|\nabla\ut\|_{L^2(\Omega)}+\|r^{-1}\ut\|_{L^2(\Omega)}.$$
Thus we can obtain
$$\cG_1\leq C\bigl(1+\|\ot\|_{L^{\f32}}^{\f32}\bigr)
\bigl(\|\nabla\ut\|_{L^2(\Omega)}^2+\|r^{-1}\ut\|_{L^2(\Omega)}^2\bigr).$$
While by using Lemma \ref{lemur} and Sobolev embedding theorem, we deduce
\begin{align*}
\cG_2\leq C\bigl\||\ot|^{\f34}\bigr\|_{L^2}^{\f12}
\bigl\|\nabla|\ot|^{\f34}\bigr\|_{L^2}^{\f32}\|\ot\|_{L^2}
\leq \f12\bigl\|\nabla|\ot|^{\f34}\bigr\|_{L^2}^2
+C\bigl\||\ot|^{\f34}\bigr\|_{L^2}^2\|\ot\|_{L^2}^4.
\end{align*}
As a result, we can achieve
$$\f{d}{dt}\|\ot\|_{L^{\f32}}^{\f32}
+\bigl\|\nabla|\ot|^{\f34}\bigr\|_{L^2}^2
+\bigl\|r^{-1}|\ot|^{\f34}\bigr\|_{L^2}^2
\lesssim\bigl(1+\|\ot\|_{L^{\f32}}^{\f32}\bigr)
\bigl(\|\nabla\ut\|_{L^2(\Omega)}^2+\bigl\|\f{\ut}r\bigr\|_{L^2(\Omega)}^2
+\|\ot\|_{L^2}^4\bigr).$$
Then Gronwall's inequality together with the bounds \eqref{uniot}
and \eqref{ut2L2Omega3} leads to
\begin{equation}\begin{split}\label{otL32}
\|\ot(t)\|_{L^{\f32}}^{\f32}
+&\bigl\|\nabla|\ot|^{\f34}\bigr\|_{L^2_t(L^2)}^2
+\bigl\|r^{-1}|\ot|^{\f34}\bigr\|_{L^2_t(L^2)}^2\\
&\leq\|\ot(0)\|_{L^{\f32}}^{\f32}\exp\bigl(C\|\ut(0)\|_{L^2(\Omega)}^2
\exp(C\|\ot\|_{L^4(\R^+;L^2)})\bigr),\quad\forall\ t>0.
\end{split}\end{equation}

The only remaining thing is the estimate for $\|\ot\|_{L^\infty_t(L^1(\Omega))}$.
Let us write the equation
\begin{equation}\label{otr}
\pa_t \f{\ot}{r}+(u^r\pa_r+u^z\pa_z) \f{\ot}{r}
-(\Delta+\frac 2r\pa_r)\f{\ot}{r}-\f{2\ut\pa_z\ut}{r^2}=0,
\end{equation}
Then by multiplying ${\rm sgn}(\ot)$ on both sides of \eqref{otr},
integrating the resulting equality on $\R^3$, and using the divergence-free condition
so that the convection term vanishes, we obtain
\begin{align*}
\f{d}{dt}\bigl\|\f{\ot}r\bigr\|_{L^1}
-\int_{\R^3}\Delta\f{\ot}r\cdot{\rm sgn}(\ot)\,dx
-\int_{\R^3}\f{2\pa_r}r\bigl(\f{\ot}r\bigr)\cdot{\rm sgn}(\ot)\,dx
-\int_{\R^3}\f{2\ut\pa_z\ut}{r^2}{\rm sgn}(\ot)\,dx=0.
\end{align*}
In which, by using the fact that $\ot$ vanishes at infinity, we have
$$-\int_{\R^3}\f{2\pa_r}r\bigl(\f{\ot}r\bigr)\cdot{\rm sgn}(\ot)\,dx
=4\pi\int_{\R}\f{|\ot|}r\Big|_{r=0}\,dz\geq0.$$
On the other hand, it is well-known that
$$-\int_{\R^3}\Delta f\cdot{\rm sgn}(f)\,dx\geq0$$
for any regular enough function $f$.
As a result, we can deduce
\begin{align*}
\f{d}{dt}\bigl\|\f{\ot}r\bigr\|_{L^1}
\leq\int_{\R^3}\f{2\ut\pa_z\ut}{r^2}{\rm sgn}(\ot)\,dx
\leq C\|\nabla\ut\|_{L^2(\Omega)}\bigl\|\f{\ut}r\bigr\|_{L^2(\Omega)}.
\end{align*}
Integrating in time, together with the estimate \eqref{ut2L2Omega3} gives
$$\|\ot(t)\|_{L^1(\Omega)}\leq\|\ot(0)\|_{L^1(\Omega)}
+C\|\ut(0)\|_{L^2(\Omega)}^2\exp\bigl(C\|\ot\|_{L^4(\R^+;L^2)}\bigr),
\quad\forall\ t>0.$$
Combining this and the estimates \eqref{uniot}-\eqref{otL32}
completes the proof of \eqref{thm1.1.2}.

\section{The Proof of the decay estimate \eqref{thm1.1.3}}\label{secprop1}

This section is devoted to the proof of the decay estimate \eqref{thm1.1.3}.
As we have proved in the previous section, the global solution $(\ot,\ut)$ satisfies
the uniform in time estimate \eqref{thm1.1.2}.
Hence for any $\varepsilon>0$, we can find some
positive time $t_2$ such that
\begin{equation}\label{5.1}
\ot(t_2)\in L^1(\Omega)\cap L^{\frac32}
\cap L^2,\quad\ut(t_2)\in L^2(\Omega),\andf
\|\ot\|_{L^4([t_2,\infty);L^2)}<\ve.
\end{equation}
Then we decompose $\bigl(\ot(t_2),\ut(t_2)\bigr)$ as
$\bigl(\ot_3(0),\ut_3(0)\bigr)+\bigl(\ot_4(0),\ut_4(0)\bigr)$ with
$$\bigl(\ot_4(0),\,\ut_4(0)\bigr)=\bigl(\ot(t_2)
\cdot\chi_{\{r\geq A_2\}}(r),\,\ut(t_2)\cdot\chi_{\{r\geq A_2\}}(r)\bigr),$$
where $\chi$ is the characteristic funtion,
and $A_2$ is some sufficiently large constant such that
 for the constants $C_1,~C_2$ appearing in \eqref{fixsmallcondi},
$\bigl(\ot_4(0),\ut_4(0)\bigr)$ are small enough satisfying
$$C_1\bigl(\|\ot_4(0)\|_{L^{\frac32}}+\|\ot_4(0)\|_{L^1(\Omega)}
+\|\ut_4(0)\|_{L^2(\Omega)}\bigr)< \min\bigl\{\f1{4C_2},\frac{\ve}2\bigr\}.$$
Then by Proposition \ref{globalsmall}, we know that the following Cauchy problem
\begin{equation*}
\left\{\begin{split}
& \pa_t \ot_4+(u^r_4\pa_r+u^z_4\pa_z) \ot_4-(\pa_r^2+\pa_z^2+\frac 1r\pa_r-\frac{1}{r^2})\omega^\theta_4-\f{u^r_4 \omega^\theta_4}{r}
-\f{2u^\theta_4 \pa_z u^\theta_4}{r}=0,\\
& \pa_t \ut_4+(u^r_4\pa_r+u^z_4\pa_z)\ut_4
-(\pa_r^2+\pa_z^2+\frac 1r\pa_r-\f{1}{r^2})u^\theta_4+\f{u^r_4 u^\theta_4}{r}=0,\\
& (\ot_4,\ut_4)|_{t=0}=\bigl(\ot_4(0), \ut_4(0)\bigr),
\end{split}\right.
\end{equation*}
has a unique global solution $(\ot_4,\ut_4)$ in $X_\infty$,
which remains small all the time satisfying
\begin{equation}\label{spaceot4ut4}
\|(\ot_4,\ut_4)\|_{X_\infty}<\ve.
\end{equation}
Furthermore, exactly along the same line to the proof of \eqref{ut2L2Omega3},
we can obtain
\begin{equation}\label{addiut4}
\|\nabla\ut_4\|_{L^2_tL^2(\Omega)}+\|r^{-1}\ut_4\|_{L^2_tL^2(\Omega)}
\leq\ve.
\end{equation}

On the other hand, the remainder $(\ot_3,\ut_3)(t)=(\ot,\ut)(t+t_2)-(\ot_4,\ut_4)(t)$
satisfies
\begin{equation}\label{ot3ut3}
\left\{\begin{split}
& \pa_t \ot_3+\wt u\cdot\wt\nabla\ot_3+\wt u_3\cdot\wt\nabla\ot_4
-(\D-\f{1}{r^2})\omega^\theta_3
-\frac{u^r \ot_3 +u^r_3\ot_4}{r}
-2\frac{u^\theta \pa_z u^\theta_3+u^\theta_3 \pa_z\ut_4}{r}=0, \\
& \pa_t \ut_3+\wt u\cdot\wt\nabla\ut_3+\wt u_3\cdot\wt\nabla\ut_4
-(\D-\f{1}{r^2})\ut_3
+\frac{u^r_3 u^\theta_3+\ur_3\ut_4+\ur_4\ut_3}{r}=0,\\
& (\ot_3,\ut_3)|_{t=0}
=(\ot(t_2)\cdot\chi_{\{r< A_2\}},\,\ut(t_2)\cdot\chi_{\{r< A_2\}}).
\end{split}\right.
\end{equation}
And in view of \eqref{5.1} and \eqref{spaceot4ut4},
we know that this solution $(\ot_3,\ut_3)$ satisfies
\begin{equation}\begin{split}\label{ot1ut1space}
\|\ot_3\|_{L^4(\R^+;L^2)}\leq
\|\ot\|_{L^4([t_2,\infty);L^2)}+\|\ot_4\|_{X_\infty}<2\ve.
\end{split}\end{equation}
Moreover, by definition, we know that
the supports of both $\ot_3(0)$ and $\ut_3(0)$ belong to $\{r\leq A_2\}$,
which together with \eqref{5.1} implies
\begin{equation}\label{ot1ut1finiteenergy}
\ot_3(0)\in L^1,\quad
r\ot_3(0)\in L^2,\quad \ut_3(0)\in L^2.
\end{equation}
We mention that these norms depend on $A_2$,
which may be large but still finite.
\smallskip

Now by multiplying ${\rm sgn}(\ot_3)$ on both sides of the equation
of $\ot_3$ in \eqref{ot3ut3}, and then integrating the resulting equality on $\R^3$,
we obtain
\begin{align*}
\f{d}{dt} \|\ot_3\|_{L^1}
-\int_{\R^3}\Delta\ot_3\cdot& {\rm sgn}(\ot_3)\,dx
+\int_{\R^3} r^{-2}|\ot_3|\,dx
+\int_{\R^3}(\wt u_3\cdot\wt\nabla\ot_4)\cdot {\rm sgn}(\ot_3)\,dx\\
&-\int_{\R^3}r^{-1}{\rm sgn}(\ot_3)\cdot\bigl(u^r \ot_3 +u^r_3\ot_4+2u^\theta \pa_z u^\theta_3+2u^\theta_3 \pa_z\ut_4\bigr)\,dx=0.
\end{align*}
In which, it is well-known that
$$-\int_{\R^3}\Delta\ot_3\cdot{\rm sgn}(\ot_3)\,dx\geq0.$$
As a result, we achieve
\begin{equation}\begin{split}\label{ot1L12}
\f{d}{dt}\|\ot_3\|_{L^1}+ \|r^{-2}\ot_3\|_{L^1}
\leq& \bigl\|r^{-\frac16}\widetilde{u}_3\bigr\|_{L^2}
\Bigl(\bigl\|r^{\frac16}\widetilde{\nabla}\ot_4\bigr\|_{L^2}
+\bigl\|r^{-\frac56}\ot_4\bigr\|_{L^2}\Bigr)
+\|\ot_3\|_{L^2}\|r^{-1}\ur\|_{L^{2}}\\
&+2\|\pa_z\ut_3\|_{L^2}\|r^{-1}\ut\|_{L^2}
+2\bigl\|r^{-\frac12}\ut_3\bigr\|_{L^2}\bigl\|r^{-\frac12}\pa_z\ut_4\bigr\|_{L^2}.
\end{split}\end{equation}
In which, by using Lemma \ref{lemuruz} and interpolation inequality, we obtain
$$\bigl\|r^{-\frac16}\widetilde{u}_3\bigr\|_{L^2}
=\bigl\|r^{\frac13}\widetilde{u}_3\bigr\|_{L^2(\Omega)}
\lesssim\bigl\|r^{\frac23}\ot_3\bigr\|_{L^{\frac65}(\Omega)}
=\bigl\|r^{-\frac16}\ot_3\bigr\|_{L^{\frac65}}
\leq \|r^{-2}\ot_3\|_{L^1}^{\frac1{12}}\|\ot_3\|_{L^1}^{\frac7{12}}
\|\ot_3\|_{L^2}^{\frac13}.$$
By using this together with Young's inequality, we can achieve
\begin{equation}\begin{split}\label{5.11}
\bigl\|r^{-\frac16}&\widetilde{u}_3\bigr\|_{L^2}
\Bigl(\bigl\|r^{\frac16}\widetilde{\nabla}\ot_4\bigr\|_{L^2}
+\bigl\|r^{-\frac56}\ot_4\bigr\|_{L^2}\Bigr)\\
&\leq C\|r^{-2}\ot_3\|_{L^1}^{\frac1{12}}
\|\ot_3\|_{L^1}^{\frac7{12}}\|\ot_3\|_{L^2}^{\frac13}
\cdot t^{-\frac16}
\Bigl(\bigl\|t^{\frac16}r^{\frac16}\widetilde{\nabla}\ot_4\bigr\|_{L^2}
+\bigl\|t^{\frac16}r^{-\frac13}\ot_4\bigr\|_{L^2(\Omega)}\Bigr)\\
&\leq \frac12\|r^{-2}\ot_3\|_{L^1}+\ve t^{-\frac12}\\
&\quad+C\ve^{-\frac47}\|\ot_3\|_{L^1}\|\ot_3\|_{L^2}^{\frac47}
\Bigl(\bigl\|t^{\frac16}r^{\frac16}\widetilde{\nabla}\ot_4\bigr\|_{L^2}^{\frac{12}7}
+\|\ot_4\|_{L^2(\Omega)}^{\frac87}
\bigl\|t^{\frac12}r^{-1}\ot_4\bigr\|_{L^2(\Omega)}^{\frac47}\Bigr).
\end{split}\end{equation}
While by using Lemma \ref{lemur}, there holds
\begin{equation}
\|\ot_3\|_{L^2}\|r^{-1}\ur\|_{L^2}\lesssim\|\ot_3\|_{L^2}\|\ot\|_{L^2}
\lesssim\|\ot_3\|_{L^2}^2+\|\ot_4\|_{L^2}^2.
\end{equation}
On the other hand, by using interpolation and Young's inequality, we obtain
\begin{equation}\begin{split}\label{5.13}
&\|\pa_z\ut_3\|_{L^2}\|r^{-1}\ut\|_{L^2}
+\bigl\|r^{-\frac12}\ut_3\bigr\|_{L^2}\bigl\|r^{-\frac12}\pa_z\ut_4\bigr\|_{L^2}\\
&\leq\|\pa_z\ut_3\|_{L^2}\Bigl(\|r^{-1}\ut_3\|_{L^2}
+\|r^{-1}\ut_4\|_{L^2(\Omega)}^{\frac12}\|\ut_4\|_{L^2(\Omega)}^{\frac12}\Bigr)
+\|r^{-1}\ut_3\|_{L^2}^{\frac12}\|\ut_3\|_{L^2}^{\frac12}
\|\pa_z\ut_4\|_{L^2(\Omega)}\\
&\leq\|\pa_z\ut_3\|_{L^2}^2+\|r^{-1}\ut_3\|_{L^2}^2\\
&\qquad\qquad
+\|r^{-1}\ut_4\|_{L^2(\Omega)}\|\ut_4\|_{L^2(\Omega)}
+\|\ut_3\|_{L^2}^{2}\|\pa_z\ut_4\|_{L^2(\Omega)}^{2}
+\|\pa_z\ut_4\|_{L^2(\Omega)}.
\end{split}\end{equation}
Substituting the estimates \eqref{5.11}-\eqref{5.13} into \eqref{ot1L12}, we achieve
\begin{equation}\begin{split}\label{estimateot1L1}
&\f{d}{dt}\|\ot_3\|_{L^1}+\frac12 \|r^{-2}\ot_3\|_{L^1}
-2\bigl(\|\pa_z\ut_3\|_{L^2}^2+\|r^{-1}\ut_3\|_{L^2}^2\bigr)\\
&\lesssim\ve t^{-\frac12}+\ve^{-\frac47}\|\ot_3\|_{L^1}\|\ot_3\|_{L^2}^{\frac47}
\bigl(\bigl\|t^{\frac16}r^{\frac16}\widetilde{\nabla}\ot_4\bigr\|_{L^2}^{\frac{12}7}
+\|\ot_4\|_{L^2(\Omega)}^{\frac87}
\bigl\|t^{\frac12}r^{-1}\ot_4\bigr\|_{L^2(\Omega)}^{\frac47}\bigr)\\
&\quad+\|\ot_3\|_{L^2}^2+\|\ot_4\|_{L^2}^2
+\|\ut_3\|_{L^2}^{2}\|\pa_z\ut_4\|_{L^2(\Omega)}^{2}+\|\pa_z\ut_4\|_{L^2(\Omega)}
+\|r^{-1}\ut_4\|_{L^2(\Omega)}\|\ut_4\|_{L^2(\Omega)}.
\end{split}\end{equation}

Next, by using the equation for $\ot_3$ in \eqref{ot3ut3}, we can derive the
equation for $r\ot_3$ that
\begin{equation}\label{eqtrot1}
\begin{split}
& \pa_t (r \ot_3)+\wt u\cdot\wt\nabla(r\ot_3)
+\wt u_3\cdot\wt\nabla(r \ot_4)-\Delta(r\ot_3)
+2\frac {\pa_r(r\ot_3)}r\\
& \qquad\qquad\qquad\qquad
-2\ur\ot_3-2\ur_3\ot_4-2u^\theta_3 (\pa_z\ut_3+\pa_z\ut_4)
-2u^\theta_4 \pa_z u^\theta_3=0.
\end{split}\end{equation}
Thanks to the homogeneous Dirichlet boundary condition for $\ot_3$
on $\{r=0\}$, there holds
$$\int_{\R^3} 2\frac{\pa_r(r\ot_3)}r\cdot r\ot_3\,dx
=2\pi\int_{\R}|r\ot_3(r,z)|^2\Big|_{r=\infty}\, dz\geq 0.$$
And by using integration by parts, we have
\begin{align*}
\int_{\R^3}\wt u_3\cdot\wt\nabla(r\ot_4)\cdot r\ot_3 \,dx
=-\int_{\R^3}\widetilde{u}_3\otimes(r\ot_4):\widetilde{\nabla}(r\ot_3)\,dx.
\end{align*}
In view of these, the $L^2$ energy estimate
of $r\ot_3$ in \eqref{eqtrot1} gives
\begin{equation}\begin{split}\label{rot1L2}
&\frac12\f{d}{dt}\|r\ot_3\|_{L^2}^2
+\|\nabla(r\ot_3)\|_{L^2}^2
\leq \int_{\R^3}\bigl|\widetilde{u}_3\otimes
(r\ot_4):\widetilde{\nabla}(r\ot_3)\bigr|\,dx\\
&+2\int_{\R^3}\bigl|\ur \ot_3+\ur_3\ot_4\bigr|\cdot r |\ot_3|\,dx
+2\Bigl|\int_{\R^3}\bigl(\ut_3\pa_z\ut_3+\pa_z(\ut_3\ut_4)\bigr) r\ot_3\, dx\Bigr|
\eqdef\cG_1+\cG_2+\cG_3.
\end{split}\end{equation}
For the terms on the right-hand side of \eqref{rot1L2},
we first get by using Lemma \ref{lemuruz} that
$$\|r^{\f56}\widetilde{u}_3\|_{L^{12}}
=\bigl\|r^{\f{11}{12}}\widetilde{u}_3\bigr\|_{L^{12}(\Omega)}
\lesssim\bigl\|r^{\f{13}{12}}\ot_3\bigr\|_{L^2(\Omega)}
\leq\|\ot_3\|_{L^2}^{\f5{12}} \|r\ot_3\|_{L^2}^{\f7{12}}.$$
As a result, there holds
\begin{equation}\begin{split}\label{5.17}
\cG_1&\leq\|\nabla(r\ot_3)\|_{L^2}\|\ot_4\|_{L^2}^{\frac12}
\bigl\|r^{\frac13}\ot_4\bigr\|_{L^3}^{\frac12}
\bigl\|r^{\frac56}\widetilde{u}_3\bigr\|_{L^{12}}\\
&\leq\|\nabla(r\ot_3)\|_{L^2}\|\ot_4\|_{L^2}^{\frac12}
\bigl\|r^{\frac23}\ot_4\bigr\|_{L^3(\Omega)}^{\frac12}
\|\ot_3\|_{L^2}^{\frac5{12}}\|r\ot_3\|_{L^2}^{\frac{7}{12}}\\
&\leq\frac14\|\nabla(r\ot_3)\|_{L^2}^2
+C\|r\ot_3\|_{L^2}^{2}\bigl(\|\ot_3\|_{L^2}^4+\|\ot_4\|_{L^2}^4
+\bigl\|r^{\frac23}\ot_4\bigr\|_{L^3(\Omega)}^3\bigr)
+C\|\ot_4\|_{L^2}^2.
\end{split}\end{equation}
Thanks to Sobolev embedding theorem as well as Lemma \ref{lemuruz},
we can get
\begin{equation}\begin{split}
\cG_2&
\leq \|r\ot_3\|_{L^6}\cdot\bigl(\|\ur_3\|_{L^3}+\|\ur_4\|_{L^3}\bigr)
\cdot\bigl(\|\ot_3\|_{L^2}+\|\ot_4\|_{L^2}\bigr)\\
&\leq\|\nabla(r\ot_3)\|_{L^2}
\cdot\bigl(\bigl\|r^{\frac12}\ot_3\bigr\|_{L^2}
+\|\ot_4\|_{L^{\frac32}}\bigr)\cdot\bigl(\|\ot_3\|_{L^2}+\|\ot_4\|_{L^2}\bigr)\\
&\leq\frac1{16}\|\nabla(r\ot_3)\|_{L^2}^2
+\bigl(\bigl\|r^{\frac12}\ot_3\bigr\|_{L^2}^2+\|\ot_4\|_{L^{\frac32}}^2\bigr)
\cdot\bigl(\|\ot_3\|_{L^2}^2+\|\ot_4\|_{L^2}^2\bigr)\\
&\leq\frac1{16}\|\nabla(r\ot_3)\|_{L^2}^2
+C\bigl(1+\|r\ot_3\|_{L^2}^2\|\ot_3\|_{L^2}^2+\|\ot_4\|_{L^{\frac32}}^2\bigr)
\cdot\bigl(\|\ot_3\|_{L^2}^2+\|\ot_4\|_{L^2}^2\bigr).
\end{split}\end{equation}
While by using integration by parts, we have
\begin{equation}\begin{split}\label{5.19}
\cG_3&\leq \|\pa_z\ut_3\|_{L^2}\|r\ut_3\|_{L^\infty}\|\ot_3\|_{L^2}
+\|\ut_3\ut_4\|_{L^2}\bigl\|\pa_z(r\ot_3)\bigr\|_{L^2}\\
&\leq\frac1{16}\|\pa_z\ut_3\|_{L^2}^2+\frac1{16}\bigl\|\pa_z(r\ot_3)\bigr\|_{L^2}^2
+C\|r\ut_3\|_{L^\infty}^2\Bigl(\|\ot_3\|_{L^2}^2+
\|\ut_4\|_{L^2(\Omega)}\|r^{-1}\ut_4\|_{L^2(\Omega)}\Bigr).
\end{split}\end{equation}
Now by substituting the estimates \eqref{5.17}-\eqref{5.19} into \eqref{rot1L2}, we achieve
\begin{equation}\begin{split}\label{estimaterot1L2}
\frac12\f{d}{dt}\|r&\ot_3\|_{L^2}^2
+\frac12\|\nabla(r\ot_3)\|_{L^2}^2-\frac18\|\pa_z\ut_3\|_{L^2}^2\\
&\lesssim\|r\ot_3\|_{L^2}^{2}\bigl(\|\ot_3\|_{L^2}^4+\|\ot_4\|_{L^2}^4
+\bigl\|r^{\frac23}\ot_4\bigr\|_{L^3(\Omega)}^3\bigr)\\
&\quad+\bigl(1+\|\ot_4\|_{L^{\frac32}}^2+\|r\ut_3\|_{L^\infty}^2\bigr)
\bigl(\|\ot_3\|_{L^2}^2+\|\ot_4\|_{L^2}^2
+\|\ut_4\|_{L^2(\Omega)}\|r^{-1}\ut_4\|_{L^2(\Omega)}\bigr).
\end{split}\end{equation}

Finally, applying $L^2$ estimat of $\ut_3$ in \eqref{ot3ut3} yields
\begin{equation}\label{5.20}
\frac12\f{d}{dt}\|\ut_3\|_{L^2}^2+\|\nabla\ut_3\|_{L^2}^2
+\|r^{-1}\ut_3\|_{L^2}^2=-\int_{\R^3}\Bigl(\widetilde{u}_3\cdot\widetilde{\nabla}\ut_4
+\frac{u^r_3 u^\theta_3+\ur_3\ut_4+\ur_4\ut_3}{r}\Bigr)\cdot \ut_3\,dx.
\end{equation}
By using divergence-free condition again, we can write
\begin{align*}
\Bigl|\int_{\R^3}\bigl(\widetilde{u}_3\cdot\widetilde{\nabla}\ut_4
\bigr)\cdot \ut_3\,dx\Bigr|
&=\Bigl|\int_\Omega r\wt u_3\otimes\ut_4:\wt\nabla\ut_3\,drdz\Bigr|\\
&\leq\bigl\|r^{\f12}{\nabla}\ut_3\bigr\|_{L^2(\Omega)}
\bigl\|r^{\f12}\wt u_3\bigr\|_{L^4(\Omega)}
\|\ut_4\|_{L^4(\Omega)}.
\end{align*}
In which, by using Lemma \ref{lemuruz} and Sobolev embedding theorem, we have
$$\bigl\|r^{\f12}\wt u_3\bigr\|_{L^4(\Omega)}
\lesssim\bigl\|r^{\f43}\ot_3\bigr\|_{L^3(\Omega)}
\leq\bigl\|r^{\f76}\ot_3\bigr\|_{L^6(\Omega)}^{\f12}
\bigl\|r^{\f32}\ot_3\bigr\|_{L^2(\Omega)}^{\f12}
\lesssim\|\nabla(r\ot_3)\|_{L^2}^{\f12}
\bigl\|r^{\f32}\ot_3\bigr\|_{L^2(\Omega)}^{\f12}.$$
As a result, we achieve
\begin{equation}\begin{split}\label{5.21}
\Bigl|\int_{\R^3}\bigl(\widetilde{u}_3\cdot\widetilde{\nabla}\ut_4
\bigr)\cdot \ut_3\,dx\Bigr|
&\leq C\|{\nabla}\ut_3\|_{L^2}
\|\nabla(r\ot_3)\|_{L^2}^{\f12}\|r\ot_3\|_{L^2}^{\f12}\|\ut_4\|_{L^4(\Omega)}\\
&\leq\frac18\|\nabla\ut_3\|_{L^2}^2+\f18\|\nabla(r\ot_3)\|_{L^2}^2
+C\|r\ot_3\|_{L^2}^2\|\ut_4\|_{L^4(\Omega)}^4.
\end{split}\end{equation}
Similarly, we can get
\begin{equation}\begin{split}
\Bigl|\int_{\R^3}\frac{\ur_3\ut_4}{r}\cdot \ut_3\,dx\Bigr|
&\leq\|r^{-1}\ut_3\|_{L^2}\bigl\|r^{\f12}\ur_3\bigr\|_{L^4(\Omega)}
\|\ut_4\|_{L^4(\Omega)}\\
&\leq\frac14\|r^{-1}\ut_3\|_{L^2}^2+\f18\|\nabla(r\ot_3)\|_{L^2}^2
+C\|r\ot_3\|_{L^2}^2\|\ut_4\|_{L^4(\Omega)}^4.
\end{split}\end{equation}
On the other side, thanks to Lemma \ref{lemur}, we can obtain
\begin{equation}\begin{split}\label{5.22}
\Bigl|\int_{\R^3}\Bigl(\frac{u^r_3 u^\theta_3+\ur_4\ut_3}{r}\Bigr)\cdot\ut_3\,dx\Bigr|
& \leq\|\ut_3\|_{L^6}^{\frac32}\|\ut_3\|_{L^2}^{\frac12}
\bigl(\bigl\|\frac{\ur_3}{r}\bigr\|_{L^2}+\bigl\|\frac{\ur_4}{r}\bigr\|_{L^2}\bigr)\\
&\leq\|\nabla\ut_3\|_{L^2}^{\frac32}\|\ut_3\|_{L^2}^{\frac12}
\bigl(\|\ot_3\|_{L^2}+\|\ot_4\|_{L^2}\bigr)\\
&\leq\frac18\|\nabla\ut_3\|_{L^2}^2+\|\ut_3\|_{L^2}^2
\bigl(\|\ot_3\|_{L^2}^4+\|\ot_4\|_{L^2}^4\bigr).
\end{split}\end{equation}
Substituting the estimates \eqref{5.21}-\eqref{5.22} into \eqref{5.20}, we achieve
\begin{equation}\begin{split}\label{estimateut1L2}
\f12\f{d}{dt}\|\ut_3\|_{L^2}^2+\f34\|\nabla\ut_3\|_{L^2}^2
&+\f34\|r^{-1}\ut_3\|_{L^2}^2-\f14\|\nabla(r\ot_3)\|_{L^2}^2\\
&\lesssim \|r\ot_3\|_{L^2}^2
\|\ut_4\|_{L^4(\Omega)}^4+\|\ut_3\|_{L^2}^2
\bigl(\|\ot_3\|_{L^2}^4+\|\ot_4\|_{L^2}^4\bigr).
\end{split}\end{equation}

Now by summing up $\eqref{estimateot1L1}+4\times\eqref{estimaterot1L2}
+4\times\eqref{estimateut1L2}$, we get
\begin{align*}
&\f{d}{dt}\bigl(\|\ot_3\|_{L^1}+\|r\ot_3\|_{L^2}^2+\|\ut_3\|_{L^2}^2\bigr)
+\|r^{-2}\ot_3\|_{L^1}+\|\nabla(r\ot_3)\|_{L^2}^2
+\|\nabla\ut_3\|_{L^2}^2+\|r^{-1}\ut_3\|_{L^2}^2\\
&\lesssim\ve^{-\frac47}\bigl(\|\ot_3\|_{L^1}+\|r\ot_3\|_{L^2}^2+\|\ut_3\|_{L^2}^2\bigr)
\Bigl(\|\ot_3\|_{L^2}^4+\|\ot_4\|_{L^2}^4
+\|\ot_4\|_{L^2(\Omega)}^2+\bigl\|t^{\frac12}r^{-1}\ot_4\bigr\|_{L^2(\Omega)}^2\\
&\quad+\|\pa_z\ut_4\|_{L^2(\Omega)}^{2}+\|\ut_4\|_{L^4(\Omega)}^4
+\bigl\|t^{\frac16}r^{\frac16}\widetilde{\nabla}\ot_4\bigr\|_{L^2}^{2}
+\bigl\|r^{\frac23}\ot_4\bigr\|_{L^3(\Omega)}^3\Bigr)
+\ve t^{-\frac12}+\|\pa_z\ut_4\|_{L^2(\Omega)}\\
&\quad+\bigl(1+\|\ot_4\|_{L^{\frac32}}^2+\|r\ut_3\|_{L^\infty}^2\bigr)
\bigl(\|\ot_3\|_{L^2}^2+\|\ot_4\|_{L^2}^2
+\|\ut_4\|_{L^2(\Omega)}\|r^{-1}\ut_4\|_{L^2(\Omega)}\bigr).
\end{align*}
In view of \eqref{spaceot4ut4},~\eqref{addiut4},~\eqref{ot1ut1space}
and the fact that $\ve\ll1$, we have
\begin{align*}
\ve^{-\frac47}\int_0^t\Bigl(\|\ot_3&\|_{L^2}^4+\|\ot_4\|_{L^2}^4
+\|\ot_4\|_{L^2(\Omega)}^2
+\bigl\|s^{\frac12}r^{-1}\ot_4\bigr\|_{L^2(\Omega)}^2\\
&+\|\pa_z\ut_4\|_{L^2(\Omega)}^{2}+\|\ut_4\|_{L^4(\Omega)}^4
+\bigl\|s^{\frac16}r^{\frac16}\widetilde{\nabla}\ot_4\bigr\|_{L^2}^{2}
+\bigl\|r^{\frac23}\ot_4\bigr\|_{L^3(\Omega)}^3\Bigr)\,ds
\leq C\ve^{\frac{10}7},
\end{align*}
and
\begin{align*}
\int_0^t\Bigl(\ve s^{-\frac12}&+\|\pa_z\ut_4\|_{L^2(\Omega)}
+\bigl\|r^{\frac23}\ot_4\bigr\|_{L^3(\Omega)}^{\frac32}
+\bigl(1+\|\ot_4\|_{L^{\frac32}}^2+\|r\ut_3\|_{L^\infty}^2\bigr)\\
&\times\bigl(\|\ot_3\|_{L^2}^2+\|\ot_4\|_{L^2}^2
+\|\ut_4\|_{L^2(\Omega)}\|r^{-1}\ut_4\|_{L^2(\Omega)}\bigr)\Bigr)\,ds
\leq C\ve\bigl(1+\|r\ut(0)\|_{L^\infty}^2\bigr)\cdot t^{\f12},
\end{align*}
where in the last step we have used Lemma \ref{lemrut} that
$$\|r\ut_3\|_{L^\infty_t(L^\infty)}
\leq\|r\ut\|_{L^\infty_t(L^\infty)}+\|r\ut_4\|_{L^\infty_t(L^\infty)}
\leq\|r\ut(0)\|_{L^\infty}+\|r\ut_4(0)\|_{L^\infty}\leq2\|r\ut(0)\|_{L^\infty}.$$
Then by using Gronwall's inequality, we finally achieve
\begin{align*}
\|\ot_3&(t)\|_{L^1}+\|r\ot_3(t)\|_{L^2}^2+\|\ut_3(t)\|_{L^2}^2
+\int_0^t\Bigl(\|r^{-2}\ot_3\|_{L^1}+\|\nabla(r\ot_3)\|_{L^2}^2
+\|\nabla\ut_3\|_{L^2}^2\\
&+\|r^{-1}\ut_3\|_{L^2}^2\Bigr)\,ds
\leq 2\Bigl(\|\ot_3(0)\|_{L^1}+\|r\ot_3(0)\|_{L^2}^2+\|\ut_3(0)\|_{L^2}^2
+ C\ve\bigl(1+\|r\ut(0)\|_{L^\infty}^2\bigr)\cdot t^{\f12}\Bigr).
\end{align*}
This implies in particular that, there must exist some large time $t_3$ such that
\begin{align*}
&\|\ot_3(t_3)\|_{L^1}+\|r\ot_3(t_3)\|_{L^2}^2+\|\ut_3(t_3)\|_{L^2}^2
\leq C_0 \ve t_3^{\f12},\\
\|r^{-2}\ot_3(t_3)\|_{L^1}&+\|\nabla(r\ot_3)(t_3)\|_{L^2}^2
+\|\nabla\ut_3(t_3)\|_{L^2}^2+\|r^{-1}\ut_3(t_3)\|_{L^2}^2
\leq C_0\ve t_3^{-\f12},
\end{align*}
where $C_0=4C\bigl(1+\|r\ut(0)\|_{L^\infty}^2\bigr)$.
Then interpolation gives
\begin{align*}
\|\ot_3(t_3)\|_{L^1(\Omega)\cap L^{\f32}}\leq&\|\ot_3(t_3)\|_{L^1}^{\f12}
\|r^{-2}\ot_3(t_3)\|_{L^1}^{\f12}+\|r\ot_3(t_3)\|_{L^2}^{\f23}
\|r^{-2}\ot_3(t_3)\|_{L^1}^{\f13}\leq( C_0\ve)^{\f23}+ C_0\ve,\\
&\|\ut_3(t_3)\|_{L^2(\Omega)}\leq\|\ut_3(t_3)\|_{L^2}^{\f12}
\|r^{-1}\ut_3(t_3)\|_{L^2}^{\f12}\leq( C_0\ve)^{\f12}.
\end{align*}
Combining the above two estimates together with \eqref{spaceot4ut4},
we achieve
$$\|\ot(t_2+t_3)\|_{L^1(\Omega)\cap L^{\f32}}+\|\ut(t_2+t_3)\|_{L^2(\Omega)}
\leq(C_0\ve)^{\f12}+(C_0\ve)^{\f23}+\ve(1+C_0).$$
Now by choosing $\ve$ to be sufficiently small, we can use
Proposition \ref{globalsmall} to obtain
\begin{equation}\label{otutt3}
\|\ot(t)\|_{L^1(\Omega)\cap L^{\f32}}+\|\ut(t)\|_{L^2(\Omega)}
\leq2C_1\bigl((C_0\ve)^{\f12}+(C_0\ve)^{\f23}+\ve(1+C_0)\bigr),
\quad\forall\ t>t_2+t_3.
\end{equation}
Since $\ve$ can be taken arbitrarily small, by definition,
\eqref{otutt3} means exactly that
$$\lim_{t\rightarrow\infty}\|\ot(t)\|_{L^1(\Omega)\cap L^{\f32}}=0,\andf
\lim_{t\rightarrow\infty}\|\ut(t)\|_{L^2(\Omega)}=0.$$

\appendix

\section{Global well-posedness of \eqref{otut} with small initial data}\label{appfix}

The purpose of this appendix is to present the proof of Proposition \ref{globalsmall}.
By using the semi-group \eqref{lemexpresion1},
we can reformulate the systems \eqref{otut} as
\begin{equation}\label{inteeqt}
\left\{
\begin{split}
& \ot(t)=S(t)\ot_0-\int_0^t S(t-s)\Bigl(\dive_*\bigl(\widetilde{u}(s)\ot(s)\bigr)-\frac{\pa_z|\ut(s)|^2}{r}\Bigr)\,ds,\quad t>0, \\
& \ut(t)=S(t)\ut_0-\int_0^t S(t-s)\Bigl(\dive_*\bigl(\widetilde{u}(s)\ut(s)\bigr)+\frac{2\ut(s)\ur(s)}{r}\Bigr)ds,\quad t>0.
\end{split}
\right.
\end{equation}
Here and all in that follows, we always denote $\dive_*(f^r e_r+f^z e_z)
\eqdef\pa_r f^r+\pa_z f^z$.

The main idea is to apply the following fixed-point argument
to the integral formula \eqref{inteeqt}.
\begin{lem}[Lemma $5.5$ in \cite{BCD}]\label{fixpointtheorem}
{\sl Let $E$ be a Banach space, $\frak{B}(\cdot,\cdot)$ a continuous bilinear map from $E\times E$ to $E,$ and $\frak{\alpha}$ a positive
real number such that
$$\frak{\alpha}<\f1{4\|\frak{B}\|}\with \|\frak{B}\|\eqdef\sup_{\|f\|,\|g\|\leq 1}\|\frak{B}(f,g)\|.$$
Then for any $a$ in the ball $B(0,\frak{\alpha})$ in $E,$  there exists a unique $x$ in $B(0,2\frak{\alpha})$ such that
$$x=a+\frak{B}(x,x).$$
}\end{lem}

\begin{proof}[The proof of Proposition \ref{globalsmall}]
 $\bullet$ \textbf{The estimate for the linear part in \eqref{inteeqt}.}
First, thanks to \eqref{semi3} with $(\alpha,\beta)=(-1,1)$, \eqref{semi4} and
\eqref{semi5} with $\gamma=\epsilon=0$, there holds
\begin{equation}\begin{split}\label{linear1}
\|S(t)\ot_0&\|_{L^1(\Omega)}
+t\|S(t)\ot_0\|_{L^{\infty}}
+\|S(t)\ut_0\|_{L^2(\Omega)}+t^{\frac12}\|S(t)\ut_0\|_{L^{\infty}}
+t^{\frac12}\|\widetilde{\nabla}S(t)\ut_0\|_{L^2(\Omega)}\\
&\qquad+t^{\frac12}\|r^{-1}S(t)\ut_0\|_{L^2(\Omega)}
+t^{\frac34}\|r^{-1}S(t)\ut_0\|_{L^4(\Omega)}
\lesssim\|\ot_0\|_{L^1(\Omega)}+\|\ut_0\|_{L^2(\Omega)}.
\end{split}\end{equation}
Next, by using \eqref{semi3} with $(\alpha,\beta)=(\frac12,\frac13)$
and $(\frac23,\frac13)$, we obtain
\begin{equation}\begin{split}
t^{\frac14}\|S(t)&\ot_0\|_{L^2}+\|S(t)\ot_0\|_{L^{\frac32}}
+t^{\frac13}\bigl\|r^{\frac23}S(t)\ot_0\bigr\|_{L^3(\Omega)}\\
&=t^{\frac14}\|r^{\frac12}S(t)\ot_0\|_{L^2(\Omega)}+\|r^{\frac23}S(t)\ot_0\|_{L^{\frac32}(\Omega)}
+t^{\frac13}\bigl\|r^{\frac23}S(t)\ot_0\bigr\|_{L^3(\Omega)}
\lesssim \|\ot_0\|_{L^{\frac32}}.
\end{split}\end{equation}

For the space-time integral terms,
the key idea is to use Minkowski's inequality.
Precisely, by using \eqref{semi6} with
$\alpha=0,~\beta=\frac13$ and change of variables, we have
\begin{equation}\begin{split}
\bigl\| S(t)\ot_0\bigr\|_{L_T^2 L^2(\Omega)}
&=\Bigl\|\frac{1}{4\pi t}
\int_\Omega \frac{1}{r^{\frac 12}\bar{r}^{\frac 12-\frac13}}H\bigl(\frac{t}{r\bar{r}}\bigr)
\exp\Bigl(-\frac{|\zeta|^2}{4t}\Bigr)
\cdot \bar{r}^{\frac23}\ot_0(\bar{r},\bar{z})\,d\bar{r}\,d\bar{z}\Bigr\|_{L_T^2 L^2(\Omega)}\\
&\lesssim\Bigl\|\int_\Omega \bigl\|\frac{1}{4\pi t}\cdot\frac{1}{t^{\frac13}}
\exp\bigl(-\frac{|\zeta|^2}{5t}\bigr)\bigr\|_{L_T^2}
\cdot\bar{r}^{\frac23}\ot_0(\bar{r},\bar{z})\,d\bar{r}\,d\bar{z}\Bigr\|_{L^2(\Omega)}\\
&=\Bigl\|\int_\Omega \bigl\|\frac{1}{4\pi t^{\frac43}}
\exp\bigl(-\frac{1}{5t}\bigr)\|_{L_T^2}
\cdot\bigl(|\zeta|^{-2}\bigr)^{\frac56}
\cdot\bar{r}^{\frac23}\ot_0(\bar{r},\bar{z})\,d\bar{r}\,d\bar{z}\Bigr\|_{L^2(\Omega)}\\
&\lesssim\bigl\|{r}^{\frac23}\ot_0\bigr\|_{L^{\frac32}(\Omega)}
\bigl\|\bigl(\frac{1}{r^2+z^2}\bigr)^{\frac56}\bigr\|_{L^{\frac65,\infty}(\Omega)}
\lesssim\|\ot_0\|_{L^{\frac32}},
\end{split}\end{equation}
here and in all that follows, we always denote
$|\zeta|^2\eqdef (r-\bar{r})^2+(z-\bar{z})^2,$
and we have used the fact that $\bigl\|\frac{1}{4\pi t^{\frac43}}
\exp\bigl(-\frac{1}{5t}\bigr)\|_{L_T^2}$ is uniformly bounded
by some constant independent of $T$.

Exactly along the same line, using \eqref{semi6} with $\alpha=0,~\beta=1$ gives
\begin{equation}\begin{split}
\bigl\| S(t)\ut_0\bigr\|_{L_T^4 L^4(\Omega)}
&\lesssim\Bigl\|\int_\Omega \bigl\|\frac{1}{4\pi t}
\exp\bigl(-\frac{|\zeta|^2}{5t}\bigr)\|_{L_T^4}
\cdot\ut_0(\bar{r},\bar{z})\,d\bar{r}\,d\bar{z}\Bigr\|_{L^4(\Omega)}\\
&\lesssim \|\ut_0\|_{L^2(\Omega)}
\bigl\|\bigl(\frac{1}{r^2+z^2}\bigr)^{\frac34}\bigr\|_{L^{\frac43,\infty}(\Omega)}
\lesssim\|\ut_0\|_{L^2(\Omega)}.
\end{split}\end{equation}
Using \eqref{semi6} with $\alpha=\frac12,~\beta=\frac13$ gives
\begin{equation}\begin{split}
\bigl\| S(t)\ot_0\bigr\|_{L_T^4 L^2}&=\Bigl\|r^{\frac 12}\frac{1}{4\pi t}
\int_\Omega \frac{1}{r^{\frac 12}\bar{r}^{\frac 12-\frac13}}H\bigl(\frac{t}{r\bar{r}}\bigr)
\exp\Bigl(-\frac{|\zeta|^2}{4t}\Bigr)
\cdot \bar{r}^{\frac23}\ot_0(\bar{r},\bar{z})\,d\bar{r}\,d\bar{z}\Bigr\|_{L_T^4 L^2(\Omega)}\\
&\lesssim\Bigl\|\int_\Omega \bigl\|\frac{1}{4\pi t}\cdot\frac{1}{t^{\frac1{12}}}
\exp\bigl(-\frac{|\zeta|^2}{5t}\bigr)\bigr\|_{L_T^4}
\cdot\bar{r}^{\frac23}\ot_0(\bar{r},\bar{z})\,d\bar{r}\,d\bar{z}\Bigr\|_{L^2(\Omega)}\\
&\lesssim\bigl\|{r}^{\frac23}\ot_0\bigr\|_{L^{\frac32}(\Omega)}
\bigl\|\bigl(\frac{1}{r^2+z^2}\bigr)^{\frac56}\bigr\|_{L^{\frac65,\infty}(\Omega)}
\lesssim\|\ot_0\|_{L^{\frac32}}.
\end{split}\end{equation}
Using \eqref{semi6} with $\alpha=-\frac35,~\beta=1$ gives
\begin{equation}\begin{split}
\bigl\|r^{-\frac35} S(t)\ut_0\bigr\|_{L_T^{\frac52}L^{\frac52}(\Omega)}
&\lesssim\Bigl\|\int_\Omega \bigl\|\frac{1}{4\pi t}\cdot\frac{1}{t^{\frac3{10}}}
\exp\bigl(-\frac{|\zeta|^2}{5t}\bigr)\bigr\|_{L_T^{\frac52}}
\cdot\ut_0(\bar{r},\bar{z})\,d\bar{r}\,d\bar{z}\Bigr\|_{L^{\frac52}(\Omega)}\\
&\lesssim \|\ut_0\|_{L^2(\Omega)}
\bigl\|\bigl(\frac{1}{r^2+z^2}\bigr)^{\frac9{10}}\bigr\|_{L^{\frac{10}9,\infty}(\Omega)}
\lesssim\|\ut_0\|_{L^2(\Omega)}.
\end{split}\end{equation}
Using \eqref{semi6} with $\alpha=0,~\beta=\frac13$ and
 $\alpha=\f23,~\beta=\frac13$ gives
\begin{equation}\begin{split}\label{3.23}
\bigl\|t^{\frac13}S(t)\ot_0&\bigr\|_{L_T^3 L^3(\Omega)}
+\bigl\|r^{\frac23}S(t)\ot_0\bigr\|_{L_T^3 L^3(\Omega)}\\
& \lesssim\Bigl\|t^{-1}
\int_\Omega \frac{t^{\f13}+r^{\f23}}{r^{\frac 12}\bar{r}^{\frac 16}}
\bigl|H\bigl(\frac{t}{r\bar{r}}\bigr)\bigr|
\exp\Bigl(-\frac{|\zeta|^2}{4t}\Bigr)
\cdot \bar{r}^{\frac23}|\ot_0(\bar{r},\bar{z})|\,d\bar{r}\,d\bar{z}\Bigr\|_{L_T^3 L^3(\Omega)}\\
&\lesssim\Bigl\|\int_\Omega \bigl\|\frac{1}{t}
\exp\bigl(-\frac{|\zeta|^2}{5t}\bigr)\bigr\|_{L_T^3}
\cdot\bar{r}^{\frac23}|\ot_0(\bar{r},\bar{z})|\,d\bar{r}\,d\bar{z}\Bigr\|_{L^3(\Omega)}\\
&\lesssim\bigl\|{r}^{\frac23}\ot_0\bigr\|_{L^{\frac32}(\Omega)}
\bigl\|\bigl(\frac{1}{r^2+z^2}\bigr)^{\frac23}\bigr\|_{L^{\frac32,\infty}(\Omega)}
\lesssim\|\ot_0\|_{L^{\frac32}}.
\end{split}\end{equation}
And using \eqref{semi6} with $\alpha=-1,~\beta=\frac13$ gives
\begin{equation}\begin{split}
\bigl\|t^{\frac12} r^{-1}S(t)\ot_0\bigr\|_{L_T^2 L^2(\Omega)}
&\lesssim\Bigl\|\int_\Omega \bigl\|\frac{1}{4\pi t}\cdot\f{1}{t^{\f13}}
\exp\bigl(-\frac{|\zeta|^2}{5t}\bigr)\bigr\|_{L_T^2}
\cdot\bar{r}^{\frac23}\ot_0(\bar{r},\bar{z})\,d\bar{r}\,d\bar{z}\Bigr\|_{L^2(\Omega)}\\
&\lesssim\bigl\|{r}^{\frac23}\ot_0\bigr\|_{L^{\frac32}(\Omega)}
\bigl\|\bigl(\frac{1}{r^2+z^2}\bigr)^{\frac56}\bigr\|_{L^{\frac65,\infty}(\Omega)}
\lesssim\|\ot_0\|_{L^{\frac32}},
\end{split}\end{equation}

To estimate the terms having derivatives, we need to use point-wise
estimates \eqref{semi9} and \eqref{semi10}. Indeed, by using \eqref{semi9}
and \eqref{semi10} with $\gamma=\frac23,~\epsilon=-\frac23$, we deduce
\begin{equation}\begin{split}
\bigl\|t^{\frac16}r^{\frac16}\widetilde{\nabla}& S(t)\ot_0\bigr\|_{L_T^2 L^2}
=\Bigl\|t^{\frac16}r^{\frac23}\frac{1}{4\pi t}
\int_\Omega \f{\bar{r}^{\frac 12-\frac23}}{r^{\f12}}
\exp\Bigl(-\frac{(r-\bar{r})^2+(z-\bar{z})^2}{4t}\Bigr)
\cdot \bar{r}^{\frac23}\ot_0(\bar{r},\bar{z})\\
&\qquad\cdot\Bigl(-\frac{t}{r^2\bar{r}}{H}'\bigl(\frac{t}{r\bar{r}}\bigr)
-\bigl(\frac{1}{2r}+\frac{r-\bar{r}}{2t}\bigr) H\bigl(\frac{t}{r\bar{r}}\bigr),
~-\frac{z-\bar{z}}{2t}{H}\bigl(\frac{t}{r\bar{r}}\bigr)\Bigr)
\,d\bar{r}\,d\bar{z}\Bigr\|_{L_T^2 L^2(\Omega)}\\
&\lesssim\Bigl\|\int_\Omega \bigl\|t^{\frac16}\cdot\frac{1}{4\pi t}\cdot\frac{1}{t^{\frac12}}
\exp\bigl(-\frac{|\zeta|^2}{5t}\bigr)\bigr\|_{L_T^2}
\cdot\bar{r}^{\frac23}\ot_0(\bar{r},\bar{z})\,d\bar{r}\,d\bar{z}\Bigr\|_{L^2(\Omega)}\\
&\lesssim\bigl\|{r}^{\frac23}\ot_0\bigr\|_{L^{\frac32}(\Omega)}
\bigl\|\bigl(\frac{1}{r^2+z^2}\bigr)^{\frac56}\bigr\|_{L^{\frac65,\infty}(\Omega)}
\lesssim\|\ot_0\|_{L^{\frac32}}.
\end{split}\end{equation}
Similarly, by using \eqref{semi9} and \eqref{semi10} with $\gamma=\epsilon=0$,
we obtain
\begin{equation}\begin{split}\label{linear2}
\bigl\|t^{\frac16}\widetilde{\nabla}S(t)\ut_0\bigr\|_{L_T^{\frac{12}5}L^{\frac{12}5}(\Omega)}
&\lesssim\Bigl\|\int_\Omega \bigl\|t^{\frac16}\cdot\frac{1}{4\pi t}\cdot\frac{1}{t^{\frac12}}
\exp\bigl(-\frac{|\zeta|^2}{5t}\bigr)\bigr\|_{L_T^{\frac{12}5}}
\cdot\ut_0(\bar{r},\bar{z})\,d\bar{r}\,d\bar{z}\Bigr\|_{L^{\frac{12}5}(\Omega)}\\
&\lesssim\|\ut_0\|_{L^2(\Omega)}
\bigl\|\bigl(\frac{1}{r^2+z^2}\bigr)^{\frac{11}{12}}\bigr\|_{L^{\frac{12}{11},\infty}(\Omega)}
\lesssim\|\ut_0\|_{L^2(\Omega)}.
\end{split}\end{equation}

The estimates \eqref{linear1}-\eqref{linear2} guarantees the
existence of some universal constant $C_1$
such that for any $T>0$, there holds
\begin{equation}\label{estimatelinear}
\bigl\|\bigl(S(t)\ot_0,S(t)\ut_0\bigr)
\bigr\|_{X_T}\leq C_1\bigl(\|\ot_0\|_{L^{\frac32}}+\|\ot_0\|_{L^1(\Omega)}
+\|\ut_0\|_{L^2(\Omega)}\bigr).
\end{equation}

\noindent$\bullet$ \textbf{The estimate for the bilinear part in \eqref{inteeqt}.}
For any given $(\ot_1,\ut_1),~(\ot_2,\ut_2)\in X_T$, and any $t\in[0,T]$,
let us consider the bilinear map
\begin{align*}
&(\ot_1,\ut_1)\times(\ot_2,\ut_2)\mapsto
(\cF^\omega,\cF^u),\quad\text{with}\\
\mathcal{F}^\omega\bigl((\ot_1,\ut_1),(\ot_2,\ut_2)\bigr)&(t)\eqdef\int_0^t
S(t-s)\Bigl(\dive_*(\widetilde{u}_1(s)\ot_2(s))
-\frac{\pa_z(\ut_1(s)\ut_2(s))}{r}\Bigr)ds,\\
\mathcal{F}^u\bigl((\ot_1,\ut_1),(\ot_2,\ut_2)\bigr)&(t)\eqdef\int_0^tS(t-s)\Bigl(\dive_*\bigl(\widetilde{u}_1(s)\cdot\ut_2(s)\bigr)
+\frac{2\ur_1(s)\cdot\ut_2(s)}{r}\Bigr)\,ds,
\end{align*}
where $\wt{u}_i$ is the velocity field determined by
$\ot_i$ via the Biot-Savart law.
In the following, we may
abuse the notation $\|\ot\|_{X_T}=\|(\ot,0)\|_{X_T}$
and $\|\ut\|_{X_T}=\|(0,\ut)\|_{X_T}$ for convenience.

First, we deduce from \eqref{semi2} that for any $t\in]0,T]$, there holds
\begin{align*}
\|\mathcal{F}^\omega(t)\|_{L^1(\Omega)}
&\lesssim\int_0^t\frac{\|\widetilde{u}_1(s)\cdot\ot_2(s)\|_{L^1(\Omega)}}
{(t-s)^{\frac 12}}
+\frac{\|r^{-1}\ut_1(s)\ut_2(s)\|_{L^1(\Omega)} }{(t-s)^{\frac 12}}\,ds\\
&\lesssim \int_0^t\frac{s^{\f14}\|\widetilde{u}_1(s)\|_{L^4(\Omega)}
s^{\f14}\|\ot_2(s)\|_{L^{\frac 43}(\Omega)}}{(t-s)^{\frac 12}\cdot s^{\f12}}
+\frac{\|\ut_1(s)\|_{L^2(\Omega)}s^{\f12}\|r^{-1}\ut_2(s)\|_{L^2(\Omega)}}
{(t-s)^{\frac 12}\cdot s^{\f12}}\,ds\\
&\lesssim \|\ot_1\|_{X_T}\|\ot_2\|_{X_T}+\|\ut_1\|_{X_T}\|\ut_2\|_{X_T},
\end{align*}
where in the last step we have used Lemma \ref{lemuruz} to get
\begin{equation}\label{estiwtu}
s^{\f12-\f1p}\|\widetilde{u}_1(s)\|_{L^p(\Omega)}
\lesssim s^{\f12-\f1p}\|\ot_1(s)\|_{L^{\f{2p}{p+2}}(\Omega)}\lesssim\|\ot_1\|_{X_T},
\quad\forall\ p\in[2,\infty[,
\end{equation}
which will be used frequently in the following. By using
\eqref{semi2} again and \eqref{estiwtu}, we get
\begin{align*}
t\|\mathcal{F}^\omega&(t)\|_{L^\infty}
\lesssim t\int_0^{\frac t2}\frac{\|\widetilde{u}_1(s)\cdot\ot_2(s)\|_{L^1(\Omega)} }
{(t-s)^{\frac 32}}
+\frac{\|r^{-1}\ut_1(s)\ut_2(s)\|_{L^1(\Omega)} }{(t-s)^{\frac 32}}\,ds\\
&\quad+t\int_{\frac t2}^t\frac{\|\widetilde{u}_1(s)\cdot\ot_2(s)\|_{L^4(\Omega)} }
{(t-s)^{\frac 34}}
+\frac{\|r^{-1}\ut_1(s)\ut_2(s)\|_{L^4(\Omega)} }{(t-s)^{\frac 34}}\,ds\\
&\lesssim t\int_0^{\frac t2}\frac{s^{\f14}\|\widetilde{u}_1(s)\|_{L^4(\Omega)}
s^{\f14}\|\ot_2(s)\|_{L^{\frac 43}(\Omega)}}{(t-s)^{\frac 32}\cdot s^{\f12}}
+\frac{\|\ut_1(s)\|_{L^2(\Omega)}s^{\f12}\|r^{-1}\ut_2(s)\|_{L^2(\Omega)}}
{(t-s)^{\frac 32}\cdot s^{\f12}}\,ds\\
&\quad+t\int_{\frac t2}^t\frac{s^{\f14}\|\widetilde{u}_1(s)\|_{L^4(\Omega)}
s\|\ot_2(s)\|_{L^{\infty}(\Omega)}}{(t-s)^{\frac 34}\cdot s^{\f54}}
+\frac{s^{\f12}\|\ut_1(s)\|_{L^\infty(\Omega)}
s^{\f34}\|r^{-1}\ut_2(s)\|_{L^4(\Omega)}}{(t-s)^{\frac 34}\cdot s^{\f54}}\,ds\\
&\lesssim \|\ot_1\|_{X_T}\|\ot_2\|_{X_T}+\|\ut_1\|_{X_T}\|\ut_2\|_{X_T}.
\end{align*}
Similarly, by using \eqref{semi1} with $\alpha=\frac23,~\beta=-\frac23$
 and \eqref{estiwtu}, we obtain
\begin{align*}
\|\mathcal{F}^\omega&(t)\|_{L^{\frac32}}
=\bigl\|r^{\frac23}\mathcal{F}^\omega(t)\bigr\|_{L^{\frac32}(\Omega)}
\lesssim
\int_0^t\frac{\bigl\|r^{\frac23}\widetilde{u}_1(s)\cdot\ot_2(s)\bigr\|
_{L^{\frac{12}{11}}(\Omega)} }{(t-s)^{\frac 34}}
+\frac{\|r^{-\frac13}\ut_1(s)\ut_2(s)\|_{L^{\frac43}(\Omega)}}{(t-s)^{\frac 7{12}}}\,ds\\
&\lesssim \int_0^t\frac{s^{\f14}\|\widetilde{u}_1(s)\|_{L^4(\Omega)}
\|\ot_2(s)\|_{L^{\frac32}}}{(t-s)^{\frac 34}\cdot s^{\f14}}
+\frac{s^{\f14}\|\ut_1(s)\|_{L^4(\Omega)}
\bigl(s^{\f12}\|r^{-1}\ut_2(s)\|_{L^2(\Omega)}\bigr)^{\frac13}
\|\ut_2(s)\|_{L^2(\Omega)}^{\frac23}}{(t-s)^{\frac 7{12}}\cdot s^{\f5{12}}}\,ds\\
&\lesssim\|\ot_1\|_{X_T}\|\ot_2\|_{X_T}+\|\ut_1\|_{X_T}\|\ut_2\|_{X_T}.
\end{align*}
We deduce from the above three estimates that
\begin{equation}\begin{split}\label{bilinear1}
\sup_{0<t\leq T}\Bigl(\|\mathcal{F}^\omega(t)\|_{L^1(\Omega)\bigcap L^{\frac32}}
+ t\|\mathcal{F}^\omega(t)\|_{L^\infty}\Bigr)
\lesssim\|\ot_1\|_{X_T}\|\ot_2\|_{X_T}+\|\ut_1\|_{X_T}\|\ut_2\|_{X_T}.
\end{split}\end{equation}

The $\mathcal{F}^u$ can be estimated similarly. Indeed,
by using \eqref{semi2},~\eqref{semi4} and \eqref{estiwtu}, we have
$$\|\mathcal{F}^u(t)\|_{L^2(\Omega)}
\lesssim\int_0^t\frac{s^{\f14}\|\widetilde{u}_1(s)\|_{L^4(\Omega)}
s^{\f14}\|\ut_2(s)\|_{L^4(\Omega)} }{(t-s)^{\frac 12}\cdot s^{\f12}}\,ds
\lesssim\|\ot_1\|_{X_T}\|\ut_2\|_{X_T},$$
$$t^{\frac12}\|\mathcal{F}^u(t)\|_{L^{\infty}(\Omega)}\lesssim
t^{\frac12}\int_0^t\frac{s^{\f14}\|\widetilde{u}_1(s)\|_{L^4(\Omega)}
s^{\f5{12}}\|\ut_2(s)\|_{L^{12}(\Omega)} }{(t-s)^{\frac56}\cdot s^{\f23}}\,ds
\lesssim\|\ot_1\|_{X_T}\|\ut_2\|_{X_T}.$$
While thanks to \eqref{semi1} and \eqref{semi3}
with $\alpha=-1,~\beta=1$, there holds
$$t^{\frac12}\|r^{-1}\mathcal{F}^u(t)\|_{L^2(\Omega)}
\lesssim t^{\frac12}\int_0^t
\frac{s^{\f14}\|\widetilde{u}_1(s)\|_{L^4(\Omega)}
s^{\f12}\|r^{-1}\ut_2(s)\|_{L^2(\Omega)} }{(t-s)^{\frac34}\cdot s^{\f34}}\,ds
\lesssim\|\ot_1\|_{X_T}\|\ut_2\|_{X_T},$$
$$t^{\frac34}\|r^{-1}\mathcal{F}^u(t)\|_{L^4(\Omega)}
\lesssim t^{\frac34}\int_0^t
\frac{s^{\f16}\|\widetilde{u}_1(s)\|_{L^3(\Omega)}
s^{\f34}\|r^{-1}\ut_2(s)\|_{L^4(\Omega)} }{(t-s)^{\frac56}\cdot s^{\f{11}{12}}}\,ds
\lesssim\|\ot_1\|_{X_T}\|\ut_2\|_{X_T}.$$
The estimate for $t^{\frac12}\|\widetilde{\nabla}\mathcal{F}^u(t)\|_{L^2(\Omega)}$
follows from \eqref{semi5} with $\gamma=\epsilon=0$ and \eqref{estiwtu} that
\begin{align*}
t^{\frac12}\|\widetilde{\nabla}\mathcal{F}^u(t)\|_{L^2(\Omega)}
\lesssim
t^{\frac12}\int_0^t \frac{s^{\frac14}\|\widetilde{u}_1(s)\|_{L^4(\Omega)}
s^{\frac12}\bigl\|\bigl(\widetilde{\nabla} \ut_2,
r^{-1}\ut_2\bigr)(s)\bigr\|_{L^2(\Omega)}}
{(t-s)^{\frac34}\cdot s^{\frac34}}\,ds
\lesssim\|\ot_1\|_{X_T}\|\ut_2\|_{X_T}.
\end{align*}
Then we can deduce from the above estimates that for any $t\leq T$, there holds
\begin{equation}\begin{split}\label{bilinear2}
\|\mathcal{F}^u(t)\|_{L^2(\Omega)}
+t^{\frac12}&\|\mathcal{F}^u(t)\|_{L^\infty(\Omega)}
+t^{\frac12}\|r^{-1}\mathcal{F}^u(t)\|_{L^2(\Omega)}\\
&+t^{\frac12}\|\widetilde{\nabla}\mathcal{F}^u(t)\|_{L^2(\Omega)}
+t^{\frac34}\|r^{-1}\mathcal{F}^u(t)\|_{L^4(\Omega)}
\lesssim\|\ot_1\|_{X_T}\|\ut_2\|_{X_T}.
\end{split}\end{equation}

Now, let us turn to handle the space-time integral terms,
by using Hardy-Littlewood-Sobolev inequality.
First, by using \eqref{semi2}, \eqref{semi3} with
$\alpha=-\frac35,~\beta=\f35$ and \eqref{estiwtu}, we obtain
\begin{equation}\begin{split}
&\|\mathcal{F}^u\|_{L_T^4 L^4(\Omega)}
+\bigl\|r^{-\frac35}\mathcal{F}^u\bigr\|_{L_T^{\frac52}L^{\frac52}(\Omega)}\\
&\lesssim
\Bigl\|\int_0^t\frac{\|\widetilde{u}_1(s)
\ut_2(s)\|_{L^{2}(\Omega)} }{(t-s)^{\frac 34}}\,ds\Bigr\|_{L_T^4}
+\Bigl\|\int_0^t\frac{\bigl\|r^{-\frac35}\widetilde{u}_1(s)\ut_2(s)\bigr\|_{L^{\frac{20}{13}}(\Omega)}}
{(t-s)^{\frac34}}\,ds\Bigr\|_{L_T^{\frac52}}\\
&\lesssim
\|\wt u_1\|_{L_T^4L^4(\Omega)}\|\ut_2\|_{L^4_T L^{4}(\Omega)}\Bigl\|\frac{1}{|\cdot|^{\frac 34}}\Bigr\|_{L_T^{\frac43,\infty}}
+\|\wt u_1\|_{L_T^4L^4(\Omega)}
\|r^{-\frac35}\ut_2\|_{L^{\frac52}_T L^{\frac52}(\Omega)}\Bigl\|\frac{1}{|\cdot|^{\frac 34}}\Bigr\|_{L_T^{\frac43,\infty}}\\
&\lesssim\|\ot_1\|_{X_T}\|\ut_2\|_{X_T}.
\end{split}\end{equation}
While by using \eqref{semi1} with $\alpha=\beta=0$
and $\alpha=0,~\beta=-\frac25$, we get
\begin{align*}
\|\mathcal{F}^\omega\|&_{L_T^2 L^2(\Omega)}
\lesssim\Bigl\|\int_0^t\frac{\|\widetilde{u}_1(s)\|_{L^4(\Omega)}
\|\ot_2(s)\|_{L^{2}(\Omega)} }{(t-s)^{\frac 34}}\,ds
+\int_0^t \frac{\|\ut_1(s)\|_{L^4(\Omega)}
\|r^{-\frac35} \ut_2(s)\|_{L^{\frac52}(\Omega)}}
{(t-s)^{\frac{17}{20}}}\,ds\Bigr\|_{L_T^2}\\
&\lesssim\|\wt u_1\|_{L_T^4L^4(\Omega)}\|\ot_2\|_{L^2_T L^{2}(\Omega)}
\Bigl\|\frac{1}{|\cdot|^{\frac 34}}\Bigr\|_{L_T^{\frac43,\infty}}
+\|\ut_1\|_{L^4_T L^4(\Omega)}\|r^{-\frac35} \ut_2\|_{L^{\frac52}_T L^{\frac52}(\Omega)}
\Bigl\|\frac{1}{|\cdot|^{\frac{17}{20}}}\Bigr\|_{L_T^{\frac{20}{17},\infty}}\\
&\lesssim\|\ot_1\|_{X_T}\|\ot_2\|_{X_T}+\|\ut_1\|_{X_T}\|\ut_2\|_{X_T}.
\end{align*}
Thanks to \eqref{semi1} with $\alpha=\frac12,~\beta=-\frac12$
and $\alpha=\frac12,~\beta=-1$, there holds
\begin{align*}
\|\mathcal{F}^\omega(t)\|_{L_T^4 L^2}
& =\bigl\|r^{\frac12}\mathcal{F}^\omega(t)\bigr\|_{L_T^4 L^2(\Omega)}
\lesssim\Bigl\|
\int_0^t\frac{\|r^{\frac12}\widetilde{u}_1(s)\cdot\ot_2(s)\|
_{L^{\frac43}(\Omega)} }{(t-s)^{\frac 34}}
+\frac{\|\ut_1(s)\ut_2(s)\|_{L^2(\Omega)}}{(t-s)^{\frac 34}}\,ds\Bigr\|_{L_T^4}\\
&\lesssim\|\wt u_1\|_{L_T^4L^4(\Omega)}\|\ot_2\|_{L^4_T L^{2}}
\Bigl\|\frac{1}{|\cdot|^{\frac 34}}\Bigr\|_{L_T^{\frac43,\infty}}
+\|\ut_1\|_{L^4_T L^4(\Omega)}\|\ut_2\|_{L^4_T L^4(\Omega)}
\Bigl\|\frac{1}{|\cdot|^{\frac 34}}\Bigr\|_{L_T^{\frac43,\infty}}\\
&\lesssim\|\ot_1\|_{X_T}\|\ot_2\|_{X_T}+\|\ut_1\|_{X_T}\|\ut_2\|_{X_T}.
\end{align*}
And similarly
\begin{align*}
\bigl\|r^{\f23}&\mathcal{F}^\omega\|_{L_T^3 L^3(\Omega)}
\lesssim\Bigl\|\int_0^t\frac{\|\widetilde{u}_1(s)\|_{L^{\f{12}5}(\Omega)}
\bigl\|r^{\f23}\ot_2(s)\bigr\|_{L^3(\Omega)}}{(t-s)^{\f{11}{12}}}
+\f{\|\ut_1(s)\|_{L^4(\Omega)}\|\ut_2(s)\|_{L^4(\Omega)}}
{(t-s)^{\f56}}\,ds\Bigr\|_{L_T^3}\\
&\lesssim\|\ot_1\|_{L_T^{12}L^{\f{12}{11}}(\Omega)}
\bigl\|r^{\f23}\ot_2\bigr\|_{L^3_T L^3(\Omega)}
\Bigl\|\frac{1}{|\cdot|^{\f{11}{12}}}\Bigr\|_{L_T^{\f{12}{11},\infty}}
+\|\ut_1\|_{L^4_T L^4(\Omega)}\|\ut_2\|_{L^4_T L^4(\Omega)}
\Bigl\|\frac{1}{|\cdot|^{\f56}}\Bigr\|_{L_T^{\f65,\infty}}\\
&\lesssim\|\ot_1\|_{X_T}\|\ot_2\|_{X_T}+\|\ut_1\|_{X_T}\|\ut_2\|_{X_T}.
\end{align*}
Combining the above three estimates, we deduce
\begin{equation}\begin{split}\label{bilinear3}
\|\mathcal{F}^\omega\|_{L_T^2 L^2(\Omega)\bigcap L_T^4 L^2}
+\bigl\|r^{\f23}\mathcal{F}^\omega\|_{L_T^3 L^3(\Omega)}
\lesssim\|\ot_1\|_{X_T}\|\ot_2\|_{X_T}+\|\ut_1\|_{X_T}\|\ut_2\|_{X_T}.
\end{split}\end{equation}

At last, let us handle the terms having time as weight.
For $\bigl\|t^{\frac13}\mathcal{F}^\omega\bigr\|_{L^3_T L^3(\Omega)}$,
we write
\begin{equation}\begin{split}\label{omegaL33Omega}
&\qquad\qquad\qquad\bigl\|t^{\frac13}\mathcal{F}^\omega
\bigr\|_{L^3_T L^3(\Omega)}\leq
\Rmnum{1}_1+\Rmnum{1}_2,\quad\text{where}\\
\Rmnum{1}_1&=\Bigl\|t^{\frac13}\int_0^{\frac t2}
\bigl\|S(t-s)\bigl(\dive_*(\widetilde{u}_1(s)\ot_2(s))
-\frac{\pa_z(\ut_1(s)\ut_2(s))}{r}\bigr) \bigr\|_{L^3(\Omega)}\, ds\Bigr\|_{L_T^3},\\
\Rmnum{1}_2&=\Bigl\|t^{\frac13}\int_{\frac t2}^t
\bigl\|S(t-s)\bigl(\dive_*(\widetilde{u}_1(s)\ot_2(s))
-\frac{\pa_z(\ut_1(s)\ut_2(s))}{r}\bigr) \bigr\|_{L^3(\Omega)}\, ds\Bigr\|_{L_T^3}.
\end{split}\end{equation}
And one should keep in mind that,
for $s\in\bigl[0,\frac t2\bigr]$ there holds $t\thicksim t-s$,
while for $s\in\bigl[\frac t2,t\bigr]$ there holds $t\thicksim s$.
Then the first term in \eqref{omegaL33Omega} can be handled by using \eqref{semi1} that
\begin{align*}
\Rmnum{1}_1&\lesssim\Bigl\|\int_0^{\frac t2}(t-s)^{\frac13}
\Bigl(\frac{\|\widetilde{u}_1(s) \ot_2(s)\|_{L^{\frac43}(\Omega)}}{(t-s)^{\frac{11}{12}}}
+\frac{\|\ut_1(s)\ut_2(s)\|_{L^2(\Omega)}}{(t-s)^{\frac76}}\Bigr)\, ds\Bigr\|_{L_T^3}\\
&\lesssim\Bigl\|\int_0^{\frac t2}
\frac{\|\widetilde{u}_1(s)\|_{L^4(\Omega)} \|\ot_2(s)\|_{L^2(\Omega)}}{(t-s)^{\frac7{12}}}
+\frac{\|\ut_1(s)\|_{L^4(\Omega)}\|\ut_2(s)\|_{L^4(\Omega)}}{(t-s)^{\frac56}}\, ds\Bigr\|_{L_T^3}\\
&\lesssim \|\ot_1\|_{L^4_T L^{\frac43}(\Omega)} \|\ot_2\|_{L_T^2 L^2(\Omega)}
\Bigl\|\frac{1}{|\cdot|^{\frac7{12}}}\Bigr\|_{L_T^{\frac{12}7,\infty}}
+\|\ut_1\|_{L^4_T L^4(\Omega)} \|\ut_2\|_{L_T^4 L^4(\Omega)}
\Bigl\|\frac{1}{|\cdot|^{\frac 56}}\Bigr\|_{L_T^{\frac65,\infty}}\\
&\lesssim\|\ot_1\|_{X_T}\|\ot_2\|_{X_T}+\|\ut_1\|_{X_T}\|\ut_2\|_{X_T}.
\end{align*}
As for the second term in \eqref{omegaL33Omega}, in view of \eqref{semi2}, there holds
\begin{align*}
&\Rmnum{1}_2\lesssim\Bigl\|\int_{\frac t2}^t s^{\frac13}
\Bigl(\frac{\|\widetilde{u}_1(s) \ot_2(s)\|_{L^{\frac{12}7}(\Omega)}}{(t-s)^{\frac34}}
+\frac{\|r^{-\frac35}\ut_1(s)\ut_2(s)\|_{L^{\frac{30}{17}}(\Omega)}}{(t-s)^{\frac{14}{15}}}\Bigr)\, ds\Bigr\|_{L_T^3}\\
&\lesssim\Bigl\|\int_{\frac t2}^t
\frac{\|\widetilde{u}_1(s)\|_{L^4(\Omega)}\|s^{\frac13}\ot_2(s)\|_{L^3(\Omega)}}{(t-s)^{\frac34}}
+\frac{s^{\frac13}\|\ut_1(s)\|_{L^6(\Omega)}\|r^{-\frac35}\ut_2(s)\|_{L^{\frac52}(\Omega)}}{(t-s)^{\frac{14}{15}}}\, ds\Bigr\|_{L_T^3}\\
&\lesssim\|\ot_1\|_{L^4_T L^{\frac43}(\Omega)}
\bigl\|t^{\frac13}\ot_2\bigr\|_{L_T^3 L^3(\Omega)}
\Bigl\|\frac{1}{|\cdot|^{\frac 34}}\Bigr\|_{L_T^{\frac43,\infty}}
+\bigl\|t^{\frac13}\|\ut_1\|_{L^6(\Omega)}\bigr\|_{L^\infty_T}
\bigl\|r^{-\frac35}\ut_2\bigr\|_{L_T^{\frac52} L^{\frac52}(\Omega)}
\Bigl\|\frac{1}{|\cdot|^{\frac{14}{15}}}\Bigr\|_{L_T^{\frac{15}{14},\infty}}\\
&\lesssim\|\ot_1\|_{X_T}\|\ot_2\|_{X_T}+\|\ut_1\|_{X_T}\|\ut_2\|_{X_T}.
\end{align*}
Substituting the above two estimates into \eqref{omegaL33Omega} gives
\begin{equation}\label{bilinear4}
\bigl\|t^{\frac13}\mathcal{F}^\omega
\bigr\|_{L^3_T L^3(\Omega)}\lesssim
\|\ot_1\|_{X_T}\|\ot_2\|_{X_T}+\|\ut_1\|_{X_T}\|\ut_2\|_{X_T}.
\end{equation}

Similarly, we decompose the term
$\bigl\|t^{\frac12}r^{-1}\mathcal{F}^\omega\bigr\|_{L^2_T L^2(\Omega)}$
into two parts:
\begin{equation}\begin{split}\label{r-1omegaL22}
&\qquad\qquad\qquad\bigl\|t^{\frac12}r^{-1}\mathcal{F}^\omega
\bigr\|_{L^2_T L^2(\Omega)}\leq
\Rmnum{2}_1+\Rmnum{2}_2,\quad\text{where}\\
\Rmnum{2}_1&=\Bigl\|t^{\frac12}\int_0^{\frac t2}
\bigl\|r^{-1}S(t-s)\bigl(\dive_*(\widetilde{u}_1(s)\ot_2(s))
-\frac{\pa_z(\ut_1(s)\ut_2(s))}{r}\bigr) \bigr\|_{L^2(\Omega)}\, ds\Bigr\|_{L_T^2}\\
\Rmnum{2}_2&=\Bigl\| t^{\frac12}\int_{\frac t2}^t
\bigl\|r^{-1}S(t-s)\bigl(\dive_*(\widetilde{u}_1(s)\ot_2(s))
-\frac{\pa_z(\ut_1(s)\ut_2(s))}{r}\bigr) \bigr\|_{L^2(\Omega)}\, ds\Bigr\|_{L_T^2}.
\end{split}\end{equation}
By using \eqref{semi1} with $\alpha=-1,~\beta=0$, we obtain
\begin{align*}
&\Rmnum{2}_1\lesssim\Bigl\|\int_0^{\frac t2}(t-s)^{\f12}
\frac{\|\widetilde{u}_1(s)\ot_2(s)\|_{L^{\frac43}(\Omega)}
+\|r^{-1}\ut_1(s)\ut_2(s)\|_{L^{\frac43}(\Omega)}}{(t-s)^{\frac54}}
\, ds\Bigr\|_{L_T^2}\\
&\lesssim\Bigl\|\int_0^{\frac t2}
\frac{\|\widetilde{u}_1(s)\|_{L^4(\Omega)}
\| \ot_2(s)\|_{L^2(\Omega)}+\bigl\|r^{-\frac35}\ut_1(s)\bigr\|_{L^{\frac52}(\Omega)}
\bigl\|r^{-\frac35}\ut_2(s)\bigr\|_{L^{\frac52}(\Omega)}^{\frac23}
\|\ut_2(s)\|^{\frac13}_{L^4(\Omega)}}{(t-s)^{\frac34}}
\, ds\Bigr\|_{L_T^2}\\
&\lesssim\Bigl(\|\ot_1\|_{L_T^4L^{\frac43}(\Omega)}
\|\ot_2\|_{L^2_T L^2(\Omega)}+\bigl\|r^{-\frac35}\ut_1\bigr\|_{L_T^{\frac52}L^{\frac52}(\Omega)}
\bigl\|r^{-\frac35}\ut_2\bigr\|_{L_T^{\frac52}L^{\frac52}(\Omega)}^{\frac23}
\|\ut_2\|^{\frac13}_{L_T^4L^4(\Omega)}\Bigr)
\Bigl\|\frac{1}{|\cdot|^{\frac34}}\Bigr\|_{L_T^{\frac43,\infty}}\\
&\lesssim\|\ot_1\|_{X_T}\|\ot_2\|_{X_T}+\|\ut_1\|_{X_T}\|\ut_2\|_{X_T}.
\end{align*}
While by using \eqref{semi1} with $\alpha=-1,~\beta=1$ and
$\alpha=-1,~\beta=\f7{15}$,
we get
\begin{align*}
\Rmnum{2}_2
&\lesssim\Bigl\|\int_{\frac t2}^t s^{\frac12}
\frac{\|\widetilde{u}_1(s)\|_{L^4(\Omega)}
\|r^{-1} \ot_2(s)\|_{L^2(\Omega)}}{(t-s)^{\frac34}}
+s^{\frac12}\frac{\bigl\|r^{-\frac35}\ut_1(s)\bigr\|_{L^{\frac52}(\Omega)}
\bigl\|r^{-\f{13}{15}}\ut_2(s)\bigr\|_{L^{\f{10}3}(\Omega)}}
{(t-s)^{\frac{29}{30}}}\, ds\Bigr\|_{L_T^2}\\
&\lesssim\|\ot_1\|_{L_T^4L^{\frac43}(\Omega)}
\bigl\|t^{\frac12}r^{-1}\ot_2\|_{L^2_T L^2(\Omega)}\Bigl\|\frac{1}{|\cdot|^{\frac34}}\Bigr\|_{L_T^{\frac43,\infty}}\\
&\qquad\qquad+\bigl\|r^{-\frac35}\ut_1\bigr\|_{L_T^{\frac52}L^{\frac52}(\Omega)}
\bigl\|t^{\f34}\|r^{-1}\ut_2\|_{L^4(\Omega)}\bigr\|_{L_T^\infty}^{\f23}
\bigl\|r^{-\frac35}\ut_2\bigr\|_{L_T^{\frac52}L^{\frac52}(\Omega)}^{\f13}
\Bigl\|\frac{1}{|\cdot|^{\frac{29}{30}}}\Bigr\|_{L_T^{\frac{30}{29},\infty}}\\
&\lesssim\|\ot_1\|_{X_T}\|\ot_2\|_{X_T}+\|\ut_1\|_{X_T}\|\ut_2\|_{X_T}.
\end{align*}
Substituting the above two estimates into \eqref{r-1omegaL22} gives
\begin{equation}\label{bilinear5.5}
\bigl\|t^{\frac12}r^{-1}\mathcal{F}^\omega
\bigr\|_{L^2_T L^2(\Omega)}\lesssim
\|\ot_1\|_{X_T}\|\ot_2\|_{X_T}+\|\ut_1\|_{X_T}\|\ut_2\|_{X_T}.
\end{equation}

The estimates for the remaining two terms require more tricks.
We first write
\begin{equation}\begin{split}\label{nablaomegaL42}
&\qquad\qquad\qquad
\bigl\|t^{\frac16}r^{\frac16}\widetilde{\nabla}\mathcal{F}^\omega\bigr\|_{L^2_T L^2}
\leq\Rmnum{3}_1+\Rmnum{3}_2,\quad\text{where}\\
\Rmnum{3}_1&=\Bigl\|t^{\frac16}\int_0^{\frac t2}
\bigl\|r^{\frac23}\widetilde{\nabla}S(t-s)\bigl(\dive_*(\widetilde{u}_1(s)\ot_2(s))
-\frac{\pa_z(\ut_1(s)\ut_2(s))}{r}\bigr) \bigr\|_{L^2(\Omega)}\, ds\Bigr\|_{L_T^2}\\
\Rmnum{3}_2&=\Bigl\|t^{\frac16}\int_{\frac t2}^t
\bigl\|r^{\frac23}\widetilde{\nabla}S(t-s)\bigl(\dive_*(\widetilde{u}_1(s)\ot_2(s))
-\frac{\pa_z(\ut_1(s)\ut_2(s))}{r}\bigr) \bigr\|_{L^2(\Omega)}\, ds\Bigr\|_{L_T^2}.
\end{split}\end{equation}
Thanks to \eqref{semi5} with $\gamma=-\epsilon=\frac 23$
and Biot-Savart law, there holds
\begin{align*}
\Rmnum{3}_1
&\lesssim\Bigl\|\int_0^{\frac t2}
\Bigl(\frac{\bigl\|r^{\frac12}\widetilde{\nabla}\widetilde{u}_1(s)\bigr\|_{L^2(\Omega)}
\bigl\|r^{\frac16} \ot_2(s)\bigr\|_{L^2(\Omega)}}{(t-s)^{1-\frac16}}
+\frac{\bigl\|r^{\frac23}\widetilde{\nabla}\ot_2 (s)\bigr\|_{L^2(\Omega)}
\|\widetilde{u}_1(s)\|_{L^4(\Omega)}}{(t-s)^{\frac34-\frac16}}\\
&\qquad+\sum\limits_{1\leq i,j\leq 2,i\neq j}
\frac{\|\pa_z\ut_i(s)\|_{L^{\frac{12}{5}}(\Omega)}
\bigl\|r^{-\frac13} \ut_j(s)\bigr\|_{L^{3}(\Omega)}}
{(t-s)^{\frac34-\frac16}}\Bigr)\, ds\Bigr\|_{L_T^2}\\
&\lesssim \Bigl\|\int_0^{\frac t2}
\Bigl(\frac{\|\ot_1(s)\bigr\|_{L^2}\|\ot_2(s)\|_{L^2}^{\frac13}
\|\ot_2(s)\|_{L^2(\Omega)}^{\frac23}}{(t-s)^{\frac56}}
+\frac{\bigl\|s^{\frac16}r^{\frac16}\widetilde{\nabla}\ot_2(s)\bigr\|_{L^2}
\|\ot_1(s)\|_{L^{\frac43}(\Omega)}}{(t-s)^{\frac7{12}}\cdot s^{\frac16}}\\
&\qquad +\sum\limits_{1\leq i,j\leq 2,i\neq j}
\frac{\bigl\|s^{\frac16}\pa_z\ut_i(s)\bigr\|_{L^{\frac{12}{5}}(\Omega)}
\bigl\|r^{-\frac13}\ut_j(s)\bigr\|_{L^3(\Omega)}}
{(t-s)^{\frac7{12}}\cdot s^{\frac16}}\Bigr)\, ds\Bigr\|_{L_T^2}
\eqdef\Rmnum{3}_{1,1}+\Rmnum{3}_{1,2}+\Rmnum{3}_{1,3},
\end{align*}
where
$$\Rmnum{3}_{1,1}\lesssim \|\ot_1\|_{L_T^4 L^2}\|\ot_2\|_{L^4_T L^2}^{\frac13}
\|\ot_2\|_{L^2_T L^2(\Omega)}^{\frac23}
\Bigl\|\frac{1}{|\cdot|^{\frac56}}\Bigr\|_{L_T^{\frac65,\infty}}
\lesssim\|\ot_1\|_{X_T}\|\ot_2\|_{X_T},$$
and by using $t\thicksim t-s$ for $s\in[0,t/2]$ that
\begin{align*}
&\Rmnum{3}_{1,2}+\Rmnum{3}_{1,3}
\lesssim\Bigl\|\bigl(\int_0^{\frac t2}\bigl(s^{-\frac16}\bigr)^5\,ds\bigr)^{\frac15}
\Bigl(\int_0^{\frac t2}\frac{\bigl\|s^{\frac16}r^{\frac16}\widetilde{\nabla}\ot_2 (s)\bigr\|_{L^2}^{\frac54}
\|\ot_1(s)\|_{L^{\frac43}(\Omega)}^{\frac54}}{(t-s)^{\frac{35}{48}}}\\
&\qquad\qquad\qquad\qquad\qquad\qquad\qquad\qquad+\sum\limits_{i\neq j}
\frac{\bigl\|s^{\frac16}\pa_z\ut_i(s)\bigr\|_{L^{\frac{12}5}(\Omega)}^{\frac54}
\bigl\|r^{-\frac13}\ut_j(s)\bigr\|_{L^3(\Omega)}^{\frac54}}
{(t-s)^{\frac{35}{48}}}\,ds\Bigr)^{\frac45}\Bigr\|_{L_T^2}\\
&\lesssim \Bigl\|\int_0^{\frac t2}
\Bigl(\frac{\bigl\|s^{\frac16}r^{\frac16}\widetilde{\nabla}\ot_2 (s)\bigr\|_{L^2}^{\frac54}
\|\ot_1(s)\|_{L^{\frac43}(\Omega)}^{\frac54}}{(t-s)^{\frac{11}{16}}}
+\sum\limits_{i\neq j}
\frac{\bigl\|s^{\frac16}\pa_z\ut_i(s)\bigr\|_{L^{\frac{12}5}(\Omega)}^{\frac54}
\bigl\|r^{-\frac13}\ut_j(s)\bigr\|_{L^3(\Omega)}^{\frac54}}
{(t-s)^{\frac{11}{16}}}\Bigr)\,ds\Bigr\|_{L_T^{\frac85}}^{\frac45}\\
&\lesssim\Bigl( \bigl\|t^{\frac16}r^{\frac16}\widetilde{\nabla}\ot_2\bigr\|_{L_T^2 L^2}
\|\ot_1\|_{L_T^4 L^{\frac43}(\Omega)}
+\sum\limits_{i\neq j}
\bigl\|t^{\frac16}\pa_z\ut_i\bigr\|_{L_T^{\frac{12}5} L^{\frac{12}5}(\Omega)}
\bigl\|r^{-\frac13}\ut_j\bigr\|_{L_T^3 L^3(\Omega)}\Bigr)
\Bigl\|\frac{1}{|\cdot|^{\frac{11}{16}}}\Bigr\|_{L_T^{\frac{16}{11},\infty}}^{\frac45}\\
&\lesssim\|\ot_1\|_{X_T}\|\ot_2\|_{X_T}+\|\ut_1\|_{X_T}\|\ut_2\|_{X_T}.
\end{align*}
While by using \eqref{semi5} with $\gamma=-\epsilon=\frac 23$,
and the fact that $s\thicksim t$ for $s\in[t/2,t]$, we get
\begin{align*}
\Rmnum{3}_2
\lesssim\Bigl\|&\int_{\frac t2}^t
\frac{\bigl\|r^{\frac12}\widetilde{\nabla}\widetilde{u}_1(s)\bigr\|_{L^2(\Omega)}
s^{\frac14}\bigl\|r^{\frac16} \ot_2(s)\bigr\|_{L^{\frac83}(\Omega)}}{(t-s)^{\frac78}s^{\frac{1}{12}}}
+\frac{\bigl\|s^{\frac16}r^{\frac23}\widetilde{\nabla}\ot_2 (s)\bigr\|_{L^2(\Omega)}
\|\widetilde{u}_1(s)\|_{L^4(\Omega)}}{(t-s)^{\frac34}}\\
&+\sum\limits_{1\leq i,j\leq 2,i\neq j}
\frac{\bigl\|s^{\frac16}\pa_z\ut_i(s)\bigr\|_{L^{\frac{12}5}(\Omega)}
\bigl\|r^{-\frac13} \ut_j(s)\bigr\|_{L^3(\Omega)}
}{(t-s)^{\frac34}}\, ds\Bigr\|_{L_T^2}\eqdef\Rmnum{3}_{2,1}
+\Rmnum{3}_{2,2}+\Rmnum{3}_{2,3}.
\end{align*}
In which, we can derive from the Biot-Savart law and
H\"{o}lder's inequality that
$$\bigl\|r^{\frac12}\widetilde{\nabla}\widetilde{u}_1\bigr\|_{L^2(\Omega)}
=\|\widetilde{\nabla}\widetilde{u}_1\bigr\|_{L^2}
\lesssim\|\ot_1\|_{L^2},\quad
s^{\frac14}\bigl\|r^{\frac16}\ot_2\bigr\|_{L^{\frac83}(\Omega)}
\leq \bigl\|s^{\frac13}\ot_2\|_{L^3(\Omega)}^{\frac58}\|\ot_2\|_{L^2}^{\frac13}
(s\|\ot_2\|_{L^\infty})^{\frac1{24}}.$$
As a result, $\Rmnum{3}_{2,1}$ can be estimated as follows:
\begin{align*}
\Rmnum{3}_{2,1}&\lesssim \Bigl\|\bigl(\int_{\frac t2}^t s^{-1}ds\bigr)^{\f{1}{12}}
\bigl(\int_{\frac t2}^t\frac{\|\ot_1(s)\|_{L^2}^{\f{12}{11}}
\bigl\|s^{\frac13}\ot_2(s)\bigr\|_{L^3(\Omega)}^{\frac{15}{22}}
\|\ot_2(s)\|_{L^2}^{\frac4{11}}\bigl(s\|\ot_2(s)\|_{L^\infty}\bigr)^{\frac1{22}}}
{(t-s)^{\frac{21}{22}}}\, ds\bigr)^{\f{11}{12}}\Bigr\|_{L_T^2}\\
&\lesssim \Bigl\|\int_{\frac t2}^t\frac{\|\ot_1(s)\|_{L^2}^{\f{12}{11}}
\bigl\|s^{\frac13}\ot_2(s)\bigr\|_{L^3(\Omega)}^{\frac{15}{22}}
\|\ot_2(s)\|_{L^2}^{\frac4{11}}\bigl(s\|\ot_2(s)\|_{L^\infty}\bigr)^{\frac1{22}}}
{(t-s)^{\frac{21}{22}}}\, ds\Bigr\|_{L_T^{\f{11}6}}^{\f{11}{12}}\\
&\lesssim \|\ot_1\|_{L_T^4L^2}
\bigl\|t^{\frac13}\ot_2\bigr\|_{L_T^3L^3(\Omega)}^{\frac58}
\|\ot_2\|_{L_T^4L^2}^{\frac13}\|t\ot_2\|_{{L_T^\infty}L^\infty}^{\frac1{24}}
\Bigl\|\frac{1}{|\cdot|^{\frac{21}{22}}}\Bigr\|_{L_T^{\frac{22}{21},\infty}}^{\f{11}{12}}\\
&\lesssim\|\ot_1\|_{X_T}\|\ot_2\|_{X_T}.
\end{align*}
And we have
\begin{align*}
\Rmnum{3}_{2,2}+\Rmnum{3}_{2,3}&\lesssim\Bigl( \bigl\|t^{\frac16}r^{\frac16}\widetilde{\nabla}\ot_2\bigr\|_{L^2_TL^2}
\|\ot_1\|_{L^4_TL^{\frac43}(\Omega)}\\
&\qquad+\sum\limits_{1\leq i,j\leq 2,i\neq j}
\bigl\|t^{\frac16}\pa_z\ut_i\bigr\|_{L^{\frac{12}5}_TL^{\frac{12}5}(\Omega)}
\bigl\|r^{-\frac13} \ut_j\bigr\|_{L_T^3L^3(\Omega)}\Bigr)
\Bigl\|\frac{1}{|\cdot|^{\frac34}}\Bigr\|_{L_T^{\frac43,\infty}}\\
&\lesssim\|\ot_1\|_{X_T}\|\ot_2\|_{X_T}+\|\ut_1\|_{X_T}\|\ut_2\|_{X_T}.
\end{align*}
Substituting all the above estimates into \eqref{nablaomegaL42} gives
\begin{equation}\label{bilinear6}
\bigl\|t^{\frac16}r^{\frac16}\widetilde{\nabla}\mathcal{F}^\omega\bigr\|_{L^2_T L^2}
\lesssim\|\ot_1\|_{X_T}\|\ot_2\|_{X_T}+\|\ut_1\|_{X_T}\|\ut_2\|_{X_T}.
\end{equation}

The last term $\bigl\|t^{\frac16}\widetilde{\nabla}
\mathcal{F}^u\bigr\|_{L^{\frac{12}5}_T L^{\frac{12}5}(\Omega)}$
can be handled similarly.
First, we write
\begin{equation}\begin{split}\label{nablau}
&\qquad\qquad\qquad\bigl\|t^{\frac16}\widetilde{\nabla}\mathcal{F}^u\bigr\|
_{L^{\frac{12}5}_T L^{\frac{12}5}(\Omega)}
\leq\Rmnum{4}_1+\Rmnum{4}_2,\quad\text{where}\\
\Rmnum{4}_1&=\Bigl\|t^{\frac16}\int_0^{\frac t2}
\bigl\|\widetilde{\nabla}S(t-s)\bigl(\widetilde{u}_1(s)\cdot\widetilde{\nabla}\ut_2(s)
+\frac{\ur_1(s)\cdot\ut_2(s)}{r}\bigr) \bigr\|_{L^{\frac{12}5}(\Omega)}\, ds\Bigr\|_{L_T^{\frac{12}5}}\\
\Rmnum{4}_2&=\Bigl\|t^{\frac16}\int_{\frac t2}^t
\bigl\|\widetilde{\nabla}S(t-s)\bigl(\widetilde{u}_1(s)\cdot\widetilde{\nabla}\ut_2(s)
+\frac{\ur_1(s)\cdot\ut_2(s)}{r}\bigr) \bigr\|_{L^{\frac{12}5}(\Omega)}\, ds\Bigr\|_{L_T^{\frac{12}5}}.
\end{split}\end{equation}
Then by using \eqref{semi5} with $\gamma=\epsilon=0$
and $\gamma=0,~\epsilon=-\frac25$, we have
$$
\Rmnum{4}_1\lesssim\Bigl\|\int_0^{\frac t2}\frac{\|\widetilde{u}_1\|_{L^4(\Omega)}
\bigl\|s^{\frac16}\widetilde{\nabla} \ut_2\bigr\|_{L^{\frac{12}5}(\Omega)}}
{(t-s)^{\frac7{12}}\cdot s^{\frac16}}
+\frac{\|\ur_1\|_{L^4(\Omega)} \bigl\|r^{-\frac35}\ut_2(s)\bigr\|_{L^{\frac52}(\Omega)}}
{(t-s)^{\frac{23}{30}}}\,ds\Bigr\|_{L_T^{\frac{12}5}}
\eqdef\Rmnum{4}_{1,1}+\Rmnum{4}_{1,2}.$$
In which, the first part can be estimated as follows
\begin{align*}
\Rmnum{4}_{1,1}&\lesssim\Bigl\|\bigl(\int_0^{\frac t2}\bigl(s^{-\frac16}\bigr)^5\,ds\bigr)^{\frac15}
\Bigl(\int_0^{\frac t2}\frac{\|\ot_1(s)\|_{L^{\frac43}(\Omega)}^{\frac54}
 \bigl\|s^{\frac16}\widetilde{\nabla}\ut_2 (s)\bigr\|_{L^{\frac{12}5}(\Omega)}^{\frac54}}
{(t-s)^{\frac{35}{48}}}\,ds\Bigr)^{\frac45}\Bigr\|_{L_T^{\frac{12}5}}\\
&\lesssim\Bigl\|\int_0^{\frac t2}\frac{\|\ot_1(s)\|_{L^{\frac43}(\Omega)}^{\frac54}
 \bigl\|s^{\frac16}\widetilde{\nabla}\ut_2 (s)\bigr\|_{L^{\frac{12}5}(\Omega)}^{\frac54}}
{(t-s)^{\frac{11}{16}}}\,ds\Bigr\|_{L_T^{\frac{48}{25}}}^{\frac45}\\
&\lesssim \|\ot_1\|_{L_T^4L^{\frac43}(\Omega)}
 \bigl\|t^{\frac16}\widetilde{\nabla}\ut_2\bigr\|_{L_T^{\frac{12}5} L^{\frac{12}5}(\Omega)}
\Bigl\|\frac{1}{|\cdot|^{\frac{11}{16}}}\Bigr\|_{L_T^{\frac{16}{11},\infty}}^{\frac45}\\
&\lesssim\|\ot_1\|_{X_T}\|\ut_2\|_{X_T}.
\end{align*}
and the second part can be estimated as
$$\Rmnum{4}_{1,2}\lesssim\|\ur_1\|_{L_T^4L^{4}(\Omega)}
\|r^{-\frac35}\ut_2\|_{L_T^{\frac52}L^{\frac52}(\Omega)}
\Bigl\|\frac{1}{|\cdot|^{\frac {23}{30}}}\Bigr\|_{L_T^{\frac{30}{23},\infty}}
\lesssim\|\ot_1\|_{X_T}\|\ut_2\|_{X_T}.$$
While by using \eqref{semi5} with $\gamma=0,~\epsilon=-\frac4{15}$, we obtain
\begin{align*}
\Rmnum{4}_2&\lesssim
\Bigl\|\int_{\frac t2}^t\frac{\|\widetilde{u}_1\|_{L^4(\Omega)}
\bigl\|s^{\frac16}\widetilde{\nabla} \ut_2\bigr\|_{L^{\frac{12}5}(\Omega)}}
{(t-s)^{\frac34}}
+\frac{\|\ur_1\|_{L^4(\Omega)}s^{\frac16}\|r^{-1}\ut_2(s)\|_{L^2(\Omega)}^{\frac13}
\bigl\|r^{-\frac35}\ut_2(s)\bigr\|_{L^{\frac52}(\Omega)}^{\frac23}}
{(t-s)^{\frac9{10}}}\,ds\Bigr\|_{L_T^{\frac{12}5}}\\
&\lesssim\|\ot_1\|_{L_T^4L^{\frac43}(\Omega)}
 \bigl\|t^{\frac16}\widetilde{\nabla}\ut_2\bigr\|_{L_T^{\frac{12}5} L^{\frac{12}5}(\Omega)}
\Bigl\|\frac{1}{|\cdot|^{\frac 34}}\Bigr\|_{L_T^{\frac43,\infty}}\\
&\qquad+\|\ot_1\|_{L_T^4L^{\frac43}(\Omega)}
\bigl\|t^{\frac12}\|r^{-1}\ut_2\|_{L^{2}(\Omega)}\bigr\|_{L_T^\infty}^{\frac13}
\bigl\|r^{-\frac35}\ut_2(s)\bigr\|_{L^{\frac52}_TL^{\frac52}(\Omega)}^{\frac23}
\Bigl\|\frac{1}{|\cdot|^{\frac9{10}}}\Bigr\|_{L_T^{\frac{10}9,\infty}}\\
&\lesssim\|\ot_1\|_{X_T}\|\ut_2\|_{X_T}.
\end{align*}
Substituting the above three estimates into \eqref{nablau}, gives
\begin{equation}\label{bilinear7}
\bigl\|t^{\frac16}\widetilde{\nabla}\mathcal{F}^u\bigr\|
_{L^{\frac{12}5}_T L^{\frac{12}5}(\Omega)}\lesssim
\|\ot_1\|_{X_T}\|\ut_2\|_{X_T}.
\end{equation}

Combining the estimates \eqref{bilinear1}-\eqref{bilinear3},
~\eqref{bilinear4},~\eqref{bilinear5.5},~\eqref{bilinear6},~\eqref{bilinear7},
finally we close the estimates
that for some universal constant $C_2$, there holds
\begin{equation}\label{estimatebilinear}
\bigl\|\bigl(\mathcal{F}^\omega,\mathcal{F}^u\bigr)
\bigr\|_{X_T}\leq C_2 \|(\ot_1,\ut_1)\|_{X_T}
\|(\ot_2,\ut_2)\|_{X_T}.
\end{equation}

\noindent$\bullet$ \textbf{Fixed-point argument}
By virtue of the estimates \eqref{estimatebilinear},~\eqref{estimatelinear},
Lemma \ref{fixpointtheorem} guarantees that
if $\|\ot_0\|_{L^{\frac32}}$, $\|\ot_0\|_{L^1(\Omega)}$,
$\|\ut_0\|_{L^2(\Omega)}$
are small enough such that
\begin{equation}\label{fixsmallcondi}
C_1\bigl(\|\ot_0\|_{L^{\frac32}}+\|\ot_0\|_{L^1(\Omega)}
+\|\ut_0\|_{L^2(\Omega)}\bigr)\leq \f1{4C_2},
\end{equation}
then \eqref{inteeqt} has a unique global solution in $X_\infty$.
Moreover, this solution remains small so that
\begin{equation*}
\|(\ot,\ut)\|_{X_\infty}\leq 2C_1\bigl(\|\ot_0\|_{L^{\frac32}}
+\|\ot_0\|_{L^1(\Omega)}+\|\ut_0\|_{L^2(\Omega)}\bigr).
\end{equation*}
This completes the proof of this proposition.
\end{proof}

\medskip

\noindent {\bf Acknowledgments.}
The author would like to thank Prof. Thierry Gallay
for his careful reading and some valuable comments.
Liu is supported by NSF of China under grant 12101053,
and the Fundamental Research Funds
for the Central Universities under Grant 310421118.

\end{document}